\documentclass[11pt]{article}
\usepackage{amsmath,amsthm,amssymb}
\usepackage[usenames,dvipsnames]{xcolor}
\usepackage{enumerate}
\usepackage{graphicx}
\usepackage{cite}
\usepackage{comment}
\usepackage{oands}
\usepackage{tikz}
\usepackage{changepage}
\usepackage{bbm}
\usepackage{mathtools}
\usepackage[margin=1in]{geometry}
\usepackage[pagewise,mathlines]{lineno}
\usepackage{appendix}
\usepackage{stmaryrd} 
\usepackage{multicol}
\usepackage{microtype}
\usepackage[colorinlistoftodos]{todonotes}

\usepackage[pdftitle={Random walk on random planar maps: spectral dimension, resistance, and displacement},
  pdfauthor={Ewain Gwynne and Jason Miller},
colorlinks=true,linkcolor=NavyBlue,urlcolor=RoyalBlue,citecolor=PineGreen,bookmarks=true,bookmarksopen=true,bookmarksopenlevel=2,unicode=true,linktocpage]{hyperref}

\theoremstyle{plain}
\newtheorem{thm}{Theorem}[section]

\newtheorem{lem}[thm]{Lemma}
\newtheorem{prop}[thm]{Proposition}

\def\@rst #1 #2other{#1}
\newcommand\MR[1]{\relax\ifhmode\unskip\spacefactor3000 \space\fi
  \MRhref{\expandafter\@rst #1 other}{#1}}
\newcommand{\MRhref}[2]{\href{http://www.ams.org/mathscinet-getitem?mr=#1}{MR#2}}

\theoremstyle{definition}
\newtheorem{defn}[thm]{Definition}
\newtheorem{remark}[thm]{Remark}

\numberwithin{equation}{section}

\newcommand{\dsb}{\begin{adjustwidth}{2.5em}{0pt}
\begin{footnotesize}}
\newcommand{\dse}{\end{footnotesize}
\end{adjustwidth}}

\newcommand{\ssb}{\begin{adjustwidth}{2.5em}{0pt}}
\newcommand{\sse}{\end{adjustwidth}}

\newcommand{\aryb}{\begin{eqnarray*}}
\newcommand{\arye}{\end{eqnarray*}}
\def\alb#1\ale{\begin{align*}#1\end{align*}}
\def\allb#1\alle{\begin{align}#1\end{align}}
\newcommand{\eqb}{\begin{equation}}
\newcommand{\eqe}{\end{equation}}
\newcommand{\eqbn}{\begin{equation*}}
\newcommand{\eqen}{\end{equation*}}

\newcommand{\BB}{\mathbbm}
\newcommand{\ol}{\overline}

\newcommand{\op}{\operatorname}

\newcommand{\frk}{\mathfrak}
\newcommand{\eqD}{\overset{d}{=}}
\newcommand{\ep}{\epsilon}
\newcommand{\rta}{\rightarrow}

\newcommand{\wt}{\widetilde}
\newcommand{\wh}{\widehat} 
\newcommand{\mcl}{\mathcal}

\newcommand{\bdy}{\partial}

\let\originalleft\left
\let\originalright\right
\renewcommand{\left}{\mathopen{}\mathclose\bgroup\originalleft}
\renewcommand{\right}{\aftergroup\egroup\originalright}

\title{Random walk on random planar maps:\\ spectral dimension, resistance, and displacement}

\date{  }
\author{
Ewain Gwynne and Jason Miller \\
\textit{University of Cambridge}
}

\begin{document}

\maketitle 

\begin{abstract}
We study simple random walk on the class of random planar maps which can be encoded by a two-dimensional random walk with i.i.d.\ increments or a two-dimensional Brownian motion via a ``mating-of-trees" type bijection. This class includes the uniform infinite planar triangulation (UIPT), the infinite-volume limits of random planar maps weighted by the number of spanning trees, bipolar orientations, or Schnyder woods they admit, and the $\gamma$-mated-CRT map for $\gamma \in (0,2)$.  For each of these maps, we obtain an upper bound for the Green's function on the diagonal, an upper bound for the effective resistance to the boundary of a metric ball, an upper bound for the return probability of the random walk to its starting point after $n$ steps, and a lower bound for the graph-distance displacement of the random walk, all of which are sharp up to polylogarithmic factors.  

When combined with work of Lee (2017), our bound for the return probability shows that the spectral dimension of each of these random planar maps is a.s.\ equal to 2, i.e., the (quenched) probability that the simple random walk returns to its starting point after $2n$ steps is $n^{-1+o_n(1)}$.  Our results also show that the amount of time that it takes a random walk to exit a metric ball is at least its volume (up to a polylogarithmic factor).  In the special case of the UIPT, this implies that random walk typically travels at least $n^{1/4 - o_n(1)}$ units of graph distance in $n$ units of time. The matching upper bound for the displacement is proven by Gwynne and Hutchcroft (2018).  These two works together resolve a conjecture of Benjamini and Curien (2013) in the UIPT case.

Our proofs are based on estimates for the mated-CRT map (which come from its relationship to SLE-decorated Liouville quantum gravity) and a strong coupling of the mated-CRT map with the other random planar map models. 
\end{abstract}


\tableofcontents

\section{Introduction}
\label{sec-intro}

\subsection{Overview} 
\label{sec-overview}

A planar map is a graph together with an embedding into the plane so that no two edges cross, viewed modulo orientation preserving homeomorphisms. The study of planar maps has a long history, going back to work of Tutte \cite{tutte} and Mullin \cite{mullin-maps} in the 1960s who worked on the question of enumerating planar maps.  In recent years, there has been a considerable focus on the large scale structure of \emph{random planar maps}.  The simple random walk is one of the most natural processes that one can put on a random planar map and there have been many recent works which study its behavior; see, e.g.,~\cite{benjamini-schramm-topology,benjamini-curien-uipq-walk,gn-recurrence,lee-conformal-growth,lee-uniformizing,gill-rohde-type,abgn-bdy,georg-poisson-bdy,angel-hyperbolic,chn-causal,gms-tutte,gms-harmonic,cg-liouville,curien-marzouk-subdiffusive}. We mention also the work~\cite{bdg-lqg-rw}, which analyzes random walk on the two-dimensional integer lattice with edge weights given by the exponential of a discrete Gaussian free field (GFF).  This model is connected to random planar maps in that they both serve as discretizations of Liouville quantum gravity (LQG), as we will discuss in more detail below.

The purpose of this paper is to resolve several open questions for the simple random walk on a certain family of infinite-volume random planar maps which arise as the Benjamini-Schramm~\cite{benjamini-schramm-topology} local limits of finite random planar maps as the size is sent to $\infty$. This family includes the uniform infinite planar triangulation (UIPT)~\cite{angel-schramm-uipt} of type II (meaning that multiple edges are allowed but self-loops are not), the infinite-volume limits of random planar maps weighted by the number of spanning trees \cite{mullin-maps,bernardi-maps,shef-burger}, bipolar orientations \cite{kmsw-bipolar}, or Schnyder woods \cite{lsw-schnyder-wood} they admit, and the mated-CRT maps, whose definition we review just below.

For each of these random planar map types, we will obtain:
\begin{itemize}
\item The spectral dimension is a.s.\ at least $2$, which when combined with the matching upper bound obtained by Lee~\cite{lee-conformal-growth,lee-uniformizing} shows that the spectral dimension is a.s.\ equal to~$2$ (Theorem~\ref{thm-map-spd}).  That is, the quenched probability that the simple random walk returns to its starting point after $2n$ steps is $n^{-1+o_n(1)}$.   This confirms a long-standing prediction in the physics literature in the setting of random planar maps, see, e.g.,~\cite{abnrw-spec-dim,aajiw-spec-dim}.
\item A lower bound for the displacement exponent of random walk, which gives the conjectured \cite{benjamini-curien-uipq-walk} exponent of $1/4$ on the UIPT (Theorem~\ref{thm-uipt-displacement}).  More generally, we show that the random walk typically travels at least $n^{1/d+ o_n(1)}$ units of graph distance in $n$ units of time, where $d$ is the ball volume growth exponent whose existence is established in~\cite{dg-lqg-dim,dzz-heat-kernel} (Theorem~\ref{thm-map-displacement}).  (The matching upper bound for the displacement exponent is proven in~\cite{gh-displacement}.)
\item A polylogarithmic upper bound for the Green's function on the diagonal and for the effective resistance between the root vertex and the boundary of a metric ball (Theorem~\ref{thm-map-green}).
\end{itemize}
See Section~\ref{sec-main-results} for precise statements.

The family of random planar maps which we consider, except for the mated-CRT maps, are those which can be bijectively encoded by means of a two-sided, two-dimensional random walk with i.i.d.\ increments via a so-called \emph{mating-of-trees} bijection. The increment distribution of the encoding walk depends on the particular model.  Roughly speaking, this bijective encoding involves constructing the two discrete random trees whose contour functions (or a slight variant thereof) are given by the two coordinates of the encoding walk, then gluing these trees together in a certain manner (the encoding is slightly different for each model). Each of the random planar maps in this family can be viewed as a discretization of $\gamma$-LQG for an appropriate value of $\gamma \in (0,2)$. The value of $\gamma$ is determined by the correlation of the coordinates of the encoding walk, which is equal to $-\cos(\pi\gamma^2/4)$. We will explain this point in more detail below.

The above bijections enable us to compare each of these random planar maps to the \emph{$\gamma$-mated-CRT map}, which is a random planar map constructed in the continuum from a pair of correlated Brownian motions with correlation $-\cos(\pi\gamma^2/4)$ via a version of the aforementioned bijections (we give the precise definition just below).  The comparison is accomplished using the results of~\cite{ghs-map-dist}, which in turn are proven by means of a strong coupling result for the encoding walk with the Brownian motions used to construct the mated-CRT map~\cite{zaitsev-kmt,kmt} (see Section~\ref{sec-coupling}).  This comparison reduces the problem of proving estimates for random walk on our given random planar map to the problem of proving estimates for random walk on the $\gamma$-mated-CRT map. Random walk on the mated-CRT map, in turn, can be analyzed by means of the relationship between mated-CRT maps and SLE-decorated LQG~\cite{wedges}, building on the estimates for harmonic functions on the mated-CRT map proven in~\cite{gms-harmonic}. See Figure~\ref{fig-cont-to-discrete-walk} for a schematic illustration of the relationships between the results and objects considered in this paper.

Our proofs use SLE/LQG theory, but can be understood with minimal knowledge of this theory provided the reader takes~\cite[Theorem~1.9]{wedges} (which relates the mated-CRT map to SLE-decorated LQG) and some estimates from~\cite{gms-harmonic} as black boxes. Note that we do \emph{not} use the convergence of the random walk on the mated-CRT map to Brownian motion as proven in~\cite{gms-tutte,gms-random-walk,bg-lbm}; rather, we just need some quantitative estimates for harmonic functions on the mated-CRT map proven in~\cite{gms-harmonic}.

\begin{figure}[ht!]
 \begin{center}
\includegraphics[scale=.8]{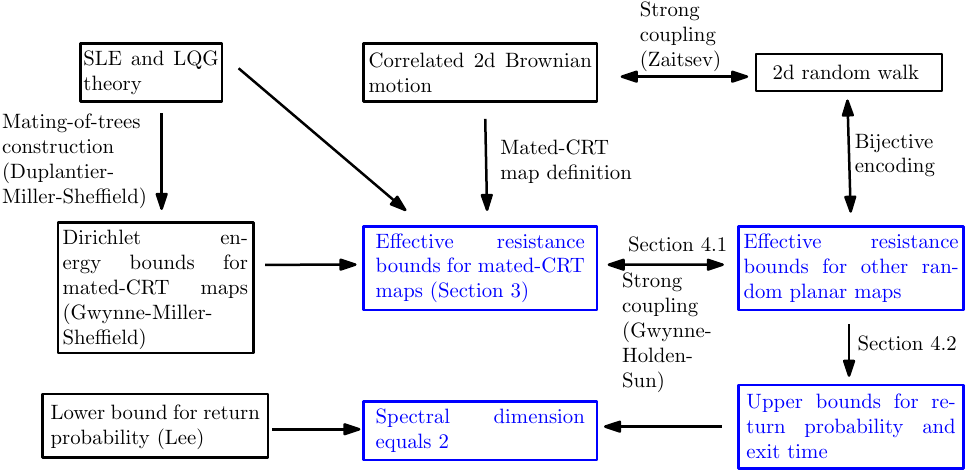} 
\caption{A schematic illustration of how the results of this paper fit together with other works (blue indicates what is accomplished in the present work). The mated-CRT map can be constructed from a two-dimensional correlated Brownian motion, or equivalently from a $\gamma$-LQG surface decorated by a space-filling SLE$_\kappa$ for $\kappa = 16/\gamma^2$ due to the results of~\cite{wedges} (c.f.\ Section~\ref{sec-peanosphere}). In the present paper, we prove estimates for effective resistance on the mated-CRT map using SLE/LQG theory (building on the Dirichlet energy bounds from~\cite{gms-harmonic}), then transfer to other random planar map models using a strong coupling result proven in~\cite{ghs-map-dist} (as explained in Section~\ref{sec-coupling}). This strong coupling result, in turn, is proven by bijectively encoding the other random planar map by means of a two-dimensional random walk, then coupling the encoding walk with the Brownian motion used to define the mated-CRT map using~\cite{zaitsev-kmt}. The effective resistance bounds for random planar maps then lead to our main results. Our upper bound for the return probability combines with the corresponding lower bound from~\cite[Theorem~1.6]{lee-uniformizing} to give that the spectral dimension of the planar maps we consider is a.s.\ equal to~$2$.
}\label{fig-cont-to-discrete-walk}
\end{center}
\end{figure}

\medskip
\noindent\textbf{Acknowledgements.} We thank Nina Holden, Asaf Nachmias, Scott Sheffield, and Xin Sun for helpful discussions. E.G.\ was partially funded by NSF grant DMS 1209044.

\subsection{Mated-CRT maps}  
\label{sec-mated-crt-map}

To define the $\gamma$-mated-CRT map for $\gamma \in (0,2)$, we start with a pair $Z = (L,R)$ of two-sided correlated Brownian motions with
\eqb \label{eqn-bm-cov}
\op{Var}(L_t) = \op{Var}(R_t) = |t| \quad \op{and} \quad \op{Cov}(L_t, R_t) = -\cos(\pi\gamma^2/4) |t| ,\quad \forall t \in \BB R.
\eqe 
Note that the correlation of $L$ and $R$ ranges over $(-1,1)$ as $\gamma$ ranges over $(0,2)$.

The mated-CRT map $\mcl G $ is the graph with vertex set $ \BB Z$, with two vertices $x_1,x_2\in  \BB Z$ with $x_1<x_2$ connected by an edge if and only if
\eqb  \label{eqn-inf-adjacency}
\left( \inf_{t\in [x_1- 1 , x_1]} L_t \right) \vee \left( \inf_{t\in [x_2- 1 , x_2]} L_t \right) \leq \inf_{t\in [x_1  , x_2 - 1]} L_t 
\eqe 
or the same holds with $R$ in place of $L$. If~\eqref{eqn-inf-adjacency} holds for both $L$ and $R$ and $|x_1-x_2|>1$, then there are two edges between $x_1$ and $x_2$. Geometrically, the condition~\eqref{eqn-inf-adjacency} means that either $|x_1-x_2| = 1$ or we can draw a horizontal line segment under the graph of $L$ with one endpoint in $[x_1- 1 , x_1]\times \BB R$ and the other endpoint in $[x_2- 1 , x_2] \times \BB R$ which intersects the graph of $L$ at its two endpoints. See Figure~\ref{fig-mated-crt-map}, which also explains how to put a planar map structure on $\mcl G $ under which it is a triangulation.

\begin{figure}[ht!]
 \begin{center}
\includegraphics[scale=.65]{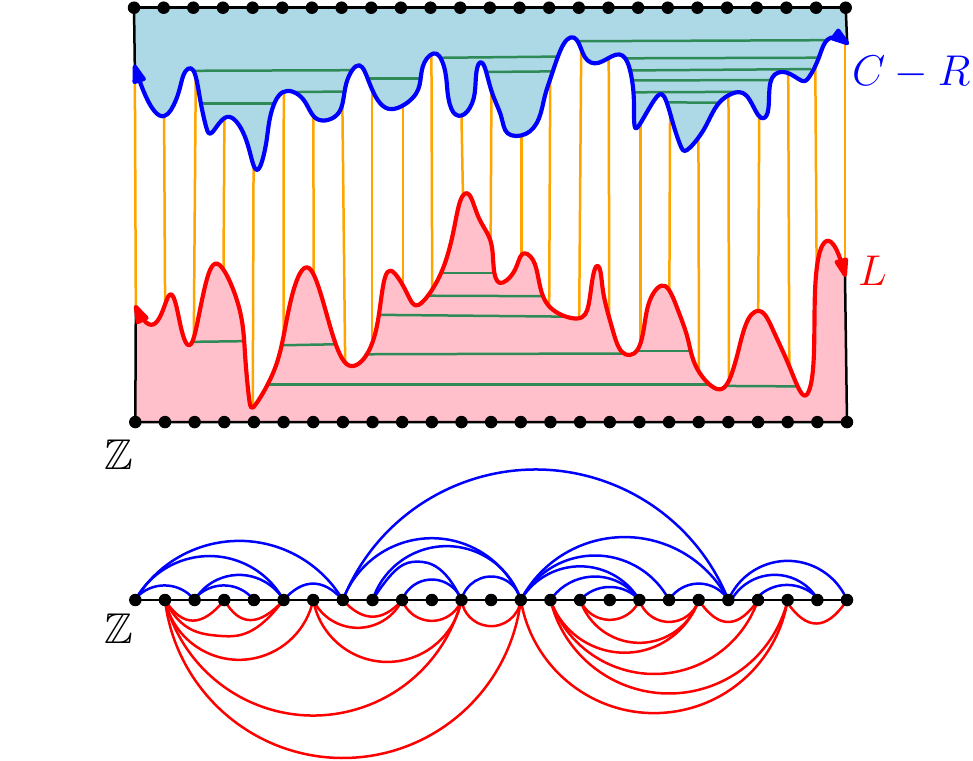}
\vspace{-0.01\textheight}
\caption{\textbf{Top:} To construct the mated-CRT map $\mcl G$ geometrically, one can draw the graph of $L$ (red) and the graph of $C-R$ (blue) for some large constant $C > 0$ chosen so that the parts of the graphs over some time interval of interest do not intersect. One then divides the region between the graphs into vertical strips (boundaries shown in orange) and identifies each strip with the horizontal coordinate $x\in \BB Z$ of its rightmost point. Vertices $x_1,x_2\in \BB Z$ are connected by an edge if and only if the corresponding strips are connected by a horizontal line segment which lies under the graph of $L$ or above the graph of $C-R$ (in particular, this happens if $|x_1-x_2| = 1$). One such segment is shown in green in the figure for each pair of vertices for which this latter condition holds.
\textbf{Bottom:} One can draw the graph $\mcl G $ in the plane by connecting two vertices $x_1,x_2 \in  \BB Z$ by an arc above (resp.\ below) the blue line if the corresponding strips are connected by a horizontal segment above (resp.\ below) the graph of $L$ (resp.\ $C-R$); and connecting each pair of consecutive vertices of $ \BB Z$ by an edge. This gives $\mcl G $ a planar map structure under which it is a triangulation. Note that the bottom picture does not correspond to the same mated-CRT map realization as the top picture.
}\label{fig-mated-crt-map}
\end{center}
\end{figure}
 
If we ignore $R$ and consider only the graph on $ \BB Z$ constructed from $L$ with the adjacency condition~\eqref{eqn-inf-adjacency}, we obtain a discretization of the continuum random tree associated with $L$~\cite{aldous-crt1,aldous-crt2,aldous-crt3}. Hence $\mcl G $ can be viewed as a discretization of the mating of two correlated CRTs. 

Mated-CRT maps are a particularly natural family of random planar maps to study.  One reason for this is that they provide bridge between discrete and continuum models (see Figure~\ref{fig-bijections} for an illustration).

To explain this, we first note that the $\gamma$-mated-CRT map is a coarse-grained approximation to many other natural random planar map models which can be obtained by gluing together a pair of discrete random trees whose contour functions have correlation $-\cos(\pi\gamma^2/4)$ in a manner directly analogous to the construction of the mated-CRT map~\cite{mullin-maps,bernardi-maps,shef-burger,kmsw-bipolar,gkmw-burger,lsw-schnyder-wood,bernardi-dfs-bijection,bhs-site-perc}. This idea is made precise in~\cite{ghs-map-dist} using the strong coupling theorem of Zaitsev~\cite{zaitsev-kmt} for random walk and Brownian motion (which is the multi-dimensional analog of~\cite{kmt}).

On the other hand, mated-CRT maps are directly connected to $\gamma$-LQG decorated by SLE$_\kappa$ for $\kappa = 16/\gamma^2$ via the peanosphere (or ``mating of trees") construction of~\cite{wedges}. We will describe this connection in Section~\ref{sec-peanosphere} below, but let us give a rough idea here. If we consider a certain type of $\gamma$-LQG surface decorated by a space-filling variant of SLE$_\kappa$ $\eta$ sampled independently from the surface and parameterized by $\gamma$-LQG mass, then the mated-CRT map can be realized as the graph of ``cells" $\eta([x-1,x])$ for $x\in\BB Z$, with two such cells considered to be adjacent if they intersect along a non-trivial connected boundary arc.
This gives us an embedding of the mated-CRT map into $\BB C$ by sending each vertex to the corresponding cell. (See Figure~\ref{fig-bijections}.)

One way to prove statements about random walk on a random planar map is to embed the map into $\BB C$ in some way, then consider how the embedded map interacts with paths and functions in $\BB C$. A number of previous papers have used this strategy with the circle packing embedding (see~\cite{stephenson-circle-packing} for an introduction) to study random walk on various random planar maps; see, e.g.,~\cite{benjamini-schramm-topology,gn-recurrence,abgn-bdy,gill-rohde-type,angel-hyperbolic,lee-conformal-growth,lee-uniformizing,cg-liouville}. In contrast, here we use the a priori embedding of the mated-CRT map which comes from space-filling SLE instead of circle packing. This allows us to prove estimates for random walk on the mated-CRT map, which we then transfer to other random planar maps using the aforementioned strong coupling results.

\begin{figure}[ht!]
 \begin{center}
\includegraphics[scale=.65]{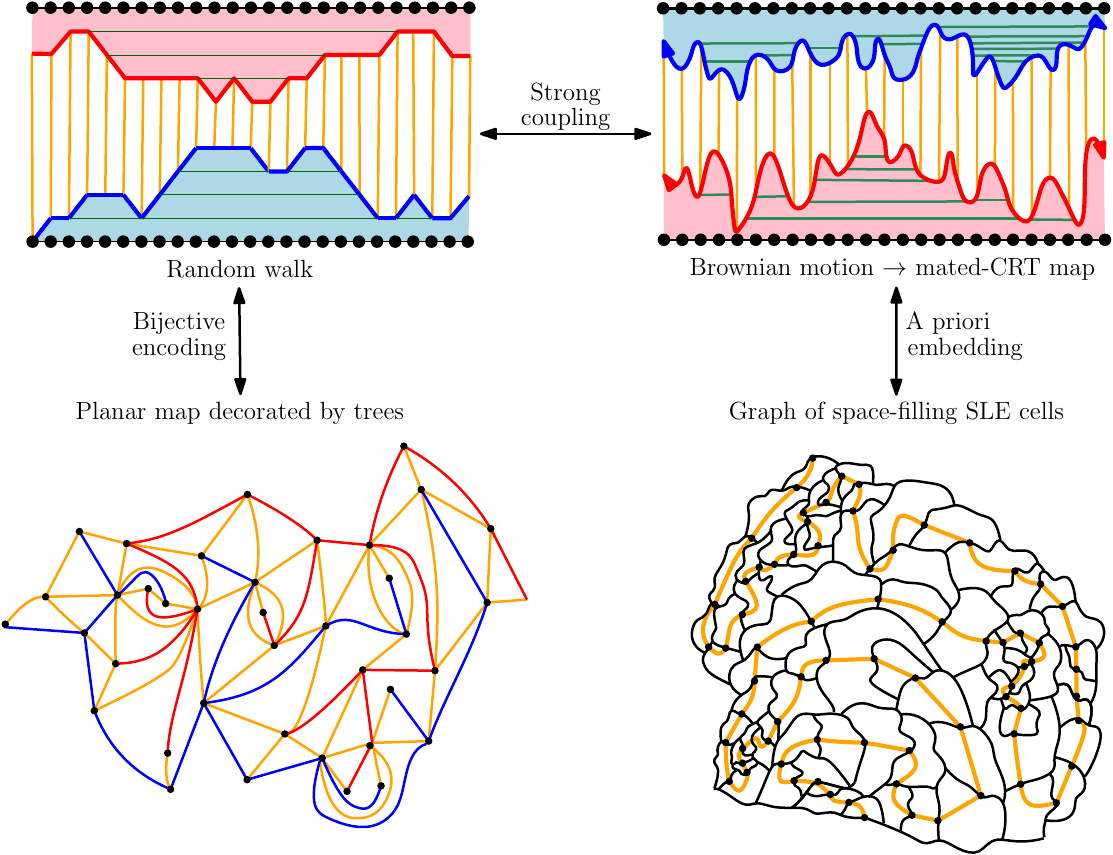} 
\caption{A visual representation of the perspective taken in this paper. On the left, we show a random planar map decorated by a pair of trees encoded by a two-dimensional random walk using a discrete version of the construction of the mated-CRT map. On the right, we show the mated-CRT map, which is defined using a pair of Brownian motions and has an a priori embedding into $\BB C$ as the adjacency graph of the ``cells" $\eta([x-1,x])$ for $x\in\BB Z$, where $\eta$ is a space-filling SLE parameterized by $\gamma$-LQG mass. Using this embedding, one can prove estimates for random walk on the mated-CRT map. One can then transfer these estimates from the mated-CRT map to other random planar maps (up to a polylogarithmic error) using a strong coupling of the Brownian motion used to define the mated-CRT map and the encoding walk for the other random planar map. 
}\label{fig-bijections}
\end{center}
\end{figure}

\subsection{Main results}
\label{sec-main-results}

Let $(M,\BB v)$ be one of the following types of infinite-volume random planar map models with its natural distinguished root vertex. In each case, we have indicated the $\gamma$-LQG universality class in parentheses.  
\begin{enumerate}
\item The \emph{uniform infinite planar triangulation} (UIPT) of type II, which is the local limit of uniform triangulations with no self-loops, but multiple edges allowed~\cite{angel-schramm-uipt} ($\gamma=\sqrt{8/3}$). 
\item The \emph{uniform infinite spanning-tree decorated planar map}, which is the local limit of random spanning-tree weighted planar maps~\cite{shef-burger,chen-fk} ($\gamma = \sqrt 2$).
\item The \emph{uniform infinite bipolar oriented planar map}, as constructed in~\cite{kmsw-bipolar}\footnote{See~\cite[Section~3.3]{ghs-map-dist} for a careful proof that the infinite-volume bipolar-oriented planar maps considered in this paper exist as Benjamini-Schramm~\cite{benjamini-schramm-topology} limits of finite bipolar-oriented maps.} ($\gamma = \sqrt{4/3}$). 
\item More generally, one of the other distributions on infinite bipolar-oriented maps considered in~\cite[Section~2.3]{kmsw-bipolar} for which the face degree distribution has an exponential tail and the correlation between the coordinates of the encoding walk is $-\cos(\pi\gamma^2/4)$ (e.g., an infinite bipolar-oriented $k$-angulation for $k\geq 3$ --- in which case $\gamma=\sqrt{4/3}$ --- or one of the bipolar-oriented maps with biased face degree distributions considered in~\cite[Remark~1]{kmsw-bipolar}, for which $\gamma \in (0,\sqrt 2)$).
\item The \emph{uniform infinite Schnyder-wood decorated triangulation}, as constructed in~\cite{lsw-schnyder-wood} ($\gamma = 1$).
\item The $\gamma$-mated-CRT map for $\gamma \in (0,2)$, as defined in Section~\ref{sec-mated-crt-map}, with $\BB v = 0$.
\end{enumerate}
As discussed in Section~\ref{sec-overview}, the reason for considering these random planar map models is that each of the first five models can be encoded by means of a two-sided, two-dimensional random walk with i.i.d.\ increments via a mating-of-trees bijection, which allows us to compare them to the mated-CRT map by means of the coupling result in~\cite{ghs-map-dist} (see~\cite{ghs-map-dist} for a review of each of these bijections). All of the results stated in this subsection also apply to other random planar map models encoded by walks with i.i.d.\ increments under such bijections, e.g., the random planar maps constructed in~\cite[Section~2]{ghs-map-dist}.

We will use that the degree of the root vertex for each of the above random planar maps has an exponential tail. This can be easily deduced from the mating-of-trees bijections and is proven in~\cite[Lemma~4.2]{angel-schramm-uipt} in the case of the UIPT,~\cite[Section 4.2]{chen-fk} in the case of the infinite spanning-tree decorated map,~\cite[Section 3.3]{ghs-map-dist} in the case of bipolar-oriented and Schnyder-wood decorated maps, and~\cite[Lemma~2.2]{gms-harmonic} in the case of the mated-CRT maps. 

Before stating our main results, we introduce the following definitions, which we will use frequently.
 
\begin{defn} \label{def-rw-law}
For a graph $G$ and a vertex $v$ of $G$, we write $\ol{\BB P}_v^{G}$ for the law of the simple random walk $X^{G}$ on $G$ started from $v$.  We write $\ol{\BB E}_v^{G}$ for the corresponding expectation. 
\end{defn}

Typically, we will take $G=M$ so that we are working with a random graph.  In this case, $\ol{\BB P}_v^M$ is the quenched law of the random walk and $\ol{\BB E}_v^M$ is the corresponding quenched expectation.

\begin{defn}
\label{def-metric-ball}
For a graph $G$, we write $\op{dist}^G(\cdot,\cdot)$ for the graph distance on $G$. For $r \geq 0$ and a vertex $v$ of $G$, we write $\mcl B_r^{G}(v)$ for the subgraph of $G$ consisting of the vertices of $G$ lying at graph distance at most $r$ from $v$ and the edges which join two such vertices. 
\end{defn}
 
Our first main result is an upper bound for the Green's function on the diagonal, both at a fixed time and at the exit time from a metric ball. For $n\in\BB N$, let $\op{Gr}_n^M(\cdot,\cdot)$ be the Green's function on $M$, i.e., $\op{Gr}_n^M(v_1,v_2)$ for vertices $v_1,v_2 \in M$  gives the (conditional given $M$) expected number of times that simple random walk on $M$ started from $v_1$ hits $v_2$ before time $n$. 

\begin{thm} \label{thm-map-green}
There exists $\alpha  , C , p>0$ (depending on the particular model) such that the following is true. For $n\in\BB N$, 
\eqb \label{eqn-map-green}
 \BB P\left[  \op{Gr}_n^M(\BB v ,\BB v ) \leq C (\log n)^{p}  \right] =   1- O_n\left( \frac{1}{(\log n)^\alpha} \right).  
\eqe
Furthermore, if we let 
\eqb \label{eqn-exit-time-def}
\sigma_r := \min\left\{ j \in \BB N : X_j^n \notin \mcl B_r^M(\BB v) \right\} ,\quad \forall r  \in \BB N
\eqe
be the exit time of the simple random walk from the ball of radius $r$ then
\eqb \label{eqn-map-eff-res0}
\BB P\left[  \op{Gr}_{\sigma_r}^M(\BB v ,\BB v ) \leq C (\log r )^{p}  \right] =   1- O_r\left( \frac{1}{(\log r)^\alpha} \right).  
\eqe
\end{thm}

Since the degree of $\BB v$ has an exponential tail, the estimate~\eqref{eqn-map-eff-res0} is equivalent to a polylogarithmic upper bound for the effective resistance from $\BB v$ to the boundary of the ball $\mcl B_r^M(\BB v)$ (see Section~\ref{sec-eff-res-def} for the definition of effective resistance).

We expect that in fact $\op{Gr}_n^M(\BB v ,\BB v )$ and $\op{Gr}_{\sigma_r}^M(\BB v ,\BB v)$ are bounded above and below by constants times $\log n$ and $\log r$, respectively, with high probability.
We prove this only in the special case of the mated-CRT maps.

\begin{thm} \label{thm-green}
For $\gamma \in (0,2)$, there exists $\alpha  = \alpha(\gamma) > 0$ and $C = C(\gamma) > 1$ such that the $\gamma$-mated-CRT map $\mcl G$ satisfies
\eqb \label{eqn-green}
 \BB P\left[ \frac{1}{C} \log n \leq \frac{ \op{Gr}_n^{\mcl G}(0,0) }{ \op{deg}^{\mcl G}\left(0  \right)} \leq C \log n  \right] =   1- O_n\left( \frac{1}{(\log n)^\alpha}  \right) ,  \quad \forall n\in\BB N 
\eqe
and
\eqb \label{eqn-eff-res0}
 \BB P\left[ \frac{1}{C} \log r \leq \frac{ \op{Gr}_{\sigma_r}^{\mcl G}(0,0) }{ \op{deg}^{\mcl G}\left(0  \right)} \leq C \log r  \right] =   1- O_r\left( \frac{1}{(\log r)^\alpha}  \right) ,  \quad \forall r \in\BB N 
\eqe 
where $\op{deg}^{\mcl G}(0)$ denotes the degree.
\end{thm}

The upper bound for the Green's function in Theorem~\ref{thm-green} improves in Theorem~\ref{thm-map-green} in the special case when $M=\mcl G$, and will be used to deduce Theorem~\ref{thm-map-green} for the other possible choices of $M$.  As we will explain, the lower bounds in Theorem~\ref{thm-green} is a straightforward consequence of~\cite[Theorem~1.4]{gms-harmonic}. We note that this lower bound implies that the simple random walk on $\mcl G$ is a.s.\ recurrent; this can also be deduced from the more general recurrence criterion of~\cite{gn-recurrence}. 
 
Our next result concerns return probabilities for random walk on $M$. 

\begin{thm}
\label{thm-map-return}
There exists $\alpha , C , p>0$ (depending on the particular model) such that for each $n\in\BB N$, it holds with probability $1-O_n( (\log n)^{-\alpha})$ that
\eqb \label{eqn-map-return}
\ol{\BB P}_{\BB v}^M\left[ X_{2n}^M  = \BB v \right] \leq \frac{C (\log n)^{p}}{n}.
\eqe
In the case when $M$ is the $\gamma$-mated-CRT map for $\gamma \in (0,2)$, one can take $p=1$. 
\end{thm}

It was proven by Lee~\cite{lee-conformal-growth,lee-uniformizing} (see in particular~\cite[Theorem~1.7]{lee-conformal-growth} or~\cite[Theorem~1.6]{lee-uniformizing}) that if $M$ is a local limit of finite planar graphs and the degree of the root vertex of $M$ has an exponential tail, then the complementary lower bound to~\eqref{eqn-map-return} holds, i.e., there is a constant $p\geq 1$ such that $ \ol{\BB P}_{\BB v}^M\left[ X_{2n}^M  = \BB v \right] \geq 1/(n (\log n)^p)$ with probability $1-O_n( (\log n)^{-1})$.\footnote{In fact, Lee considers only planar maps without multiple edges or self-loops, but it is straightforward to extend to the general case by adding a vertex to the middle of each edge of a general map to get a map without multiple edges or self-loops. We explain this carefully in Appendix~\ref{sec-return-lower}. We also note that a slightly weaker lower bound for the return probability for the planar maps considered here, with an $n^{o_n(1)}$ error rather than a polylogarithmic error, can be obtained from results of~\cite{gh-displacement}, which are proven using methods closer to those in this paper: see~\cite[Remark 3.9]{gh-displacement}. }
Each of the rooted random planar maps $(M,\BB v)$ considered here satisfies these hypotheses. 

Combining~\cite[Theorem~1.6]{lee-uniformizing} with Theorem~\ref{thm-map-return} allows us to show that the spectral dimension of $M$, which we now define, is a.s.\ equal to 2. For a graph $G$ and a vertex $v$ of $G$, the \emph{spectral dimension} is defined by 
\eqb \label{eqn-spd-def}
- 2 \lim_{n\to\infty} \frac{  \ol{\BB P}_v^G\left[ X_{2n}^{G} = x\right]}{\log n}  ,
\eqe 
if this limit exists.   It is easy to see that if the limit exists, it does not depend on the choice of $v$.
 
\begin{thm}
\label{thm-map-spd}
Almost surely, the spectral dimension of $M$ is equal to 2. More precisely, there exist constants $\alpha , p  > 0$ (depending on the particular model) such that for each $n\in\BB N$,  it holds with probability $1-O_n( (\log n)^{-\alpha})$ that
\eqb \label{eqn-map-spd-bound}
 \frac{1}{ n (\log n)^p} \leq  \ol{\BB P}_{\BB v}^M\left[ X_{2n}^M  = \BB v \right] \leq  \frac{(\log n)^p}{n } .
\eqe
\end{thm}

The estimate~\eqref{eqn-map-spd-bound} implies that the spectral dimension is a.s.\ equal to 2, even though the probability that it holds for a fixed $n$ is only $1-O_n( (\log n)^{-\alpha})$. The reason is that $n\mapsto \ol{\BB P}_{\BB v}^M\left[ X_{2n}^M  = \BB v \right]$ is non-increasing~\cite[Proposition~10.18]{markov-mixing}, so we can first prove convergence of the limit in~\eqref{eqn-spd-def} along a subsequence of the form $n_k = \exp(k^s)$ for $s > 1/\alpha$ using Borel-Cantelli, then use monotonicity to obtain the convergence for all $n$ (see Section~\ref{sec-proof}).
  
The spectral dimension is one of the most natural ways of associating a notion of dimension to a discrete fractal. This notion of dimension is important in the study of quantum gravity in the physics literature since it can be defined in a reparameterization invariant way~\cite{abnrw-spec-dim,aajiw-spec-dim}.  Theorem~\ref{thm-map-spd} is the first result to compute the spectral dimension of any planar map in the $\gamma$-LQG universality class.  We note, however, that the paper~\cite{chn-causal} shows that the spectral dimension of a \emph{causal triangulation}, a different type of random planar map which is not expected to converge to $\gamma$-LQG for any $\gamma$, is a.s.\ equal to~$2$.  There is also a purely continuum notion of the spectral dimension of $\gamma$-LQG, defined using the so-called \emph{Liouville heat kernel}, which is shown to be equal to~$2$ in~\cite{rhodes-vargas-spec-dim,andres-heat-kernel}. There are also other results concerning spectral dimensions of various graphs, e.g.,~\cite{kn-ao-conjecture} which computes the spectral dimension of the incipient infinite percolation cluster on $\BB Z^d$ for large $d$.

Our next main result gives a lower bound for the graph-distance displacement of the simple random walk in terms of the volume of a graph metric ball.

\begin{thm} \label{thm-map-displacement} 
For $r\in\BB N$, let $\sigma_r$ be the exit time from $\mcl B_r^M(\BB v)$, as in~\eqref{eqn-exit-time-def}.
There exists $\alpha > 0$, $C>0$, and $p>0$ (depending on the particular model) such that for each $r\in\BB N$, it holds with probability $1 - O_r( (\log r)^{-\alpha})$ that  
\eqb \label{eqn-map-exit-mean}
\ol{\BB E}_{\BB v}^M \left[ \sigma_r \right] \leq C (\log r)^p \# \mcl V \mcl B_r^M(\BB v )  ,
\eqe
where $\#\mcl V \mcl B_r^M(\BB v )$ denotes the number of vertices of the ball.
Furthermore, if we choose $r_n \to\infty$ such that $\BB P[\# \mcl V \mcl B_{r_n}^M(\BB v) \leq n  (\log n)^{-p}  ] \to 1$ as $n\to\infty$, then with probability tending to $1$ as $n\to\infty$, 
\eqb \label{eqn-map-exit-prob}
\ol{\BB P}_{\BB v}^M \left[ \op{dist}^M(X_n^M , \BB v )  \geq   r_n  \right] \geq 1 -   \frac{1}{C \log n} .
\eqe
\end{thm} 

In the special case of the UIPT, it is known~\cite{angel-peeling} that $ r^{-4} \# \mcl V \mcl B_r^M(\BB v ) $ is bounded above and below by powers of $(\log r)^p$ with high probability. 
One obtains a similar bound (with a polylogarithmic error) in the case of the mated-CRT map with $\gamma = \sqrt{8/3}$ via strong coupling techniques~\cite[Theorem~1.8]{ghs-map-dist}. 
Consequently, Theorem~\ref{thm-map-displacement} implies the following.

\begin{thm} \label{thm-uipt-displacement}
Suppose we are in the case when $M$ is the UIPT of type II or the $\sqrt{8/3}$-mated-CRT map and let $\sigma_r$ be as in~\eqref{eqn-exit-time-def}. There is a constant $p>0$ such that with probability tending to $1$ as $r\to\infty$, 
\eqb \label{eqn-uipt-exit-mean}
\ol{\BB E}_{\BB v}^M \left[ \sigma_r \right] \leq (\log r)^p r^4
\eqe
and with probability tending to $1$ as $n\to\infty$, 
\eqb \label{eqn-uipt-exit-prob}
\ol{\BB P}_{\BB v}^M \left[ \op{dist}^M(X_n^M , \BB v )    \geq  (\log n)^{-p} n^{1/4}   \right] \geq  \frac{1}{  \log n} .  
\eqe 
\end{thm}

In the case of the $\sqrt{8/3}$-mated-CRT map or the UIPT, Theorem~\ref{thm-uipt-displacement} gives a lower bound for the graph-distance displacement of random walk on $M$ with the conjectured exponent of $1/4$.  The matching upper bound of  $n^{1/4 + o_n(1)}$ for $\op{dist}^M(X_n^M , \BB v )$ is proven in~\cite{gh-displacement} (see also~\cite{lee-exponents} for an alternative, more recent proof).  Together, these two works prove~\cite[Conjecture~1]{benjamini-curien-uipq-walk} in the case of the UIPT.  
Prior to this the best known upper bound for the graph-distance displacement of random walk on the UIPT was $n^{1/3} (\log n)^p$ for a constant $p > 0$~\cite[Corollary 2]{benjamini-curien-uipq-walk} (see also~\cite[Theorem~1.10]{lee-conformal-growth} for an upper bound on the displacement of the walk in a more general setting which gives $n^{1/3 + o_n(1)}$ in the case of the UIPT). 

More generally, it is shown in~\cite[Theorem~1.2]{dg-lqg-dim} (building on~\cite{dzz-heat-kernel}) that for each of the random planar maps considered in this paper,  
\eqbn
d_\gamma := \lim_{r\rta 0} \frac{\log \# \mcl V \mcl B_r^M(\BB v )}{\log r}  
\eqen
exists a.s.\ and depends only on $\gamma$. Hence Theorem~\ref{thm-map-displacement} implies that the random walk on $M$ typically travels graph distance at least $n^{1/d_\gamma+o_n(1)}$ in $n$ units of time. 
It is shown in~\cite{gp-kpz} that $d_\gamma$ coincides with the Hausdorff dimension of the $\gamma$-LQG metric as constructed in~\cite{gm-uniqueness}.
Computing $d_\gamma$ for $\gamma\not=\sqrt{8/3}$ is a major open problem; see~\cite{ghs-dist-exponent,ding-goswami-watabiki,dg-lqg-dim,gm-uniqueness} and the references therein.
However, reasonably sharp upper and lower bounds for $d_\gamma$ are known~\cite{dg-lqg-dim,gp-lfpp-bounds}, which can be plugged into Theorem~\ref{thm-map-displacement} to get an explicit lower bound for graph distance displacement of the walk. 
We also note that~\cite[Theorem 1.3]{gh-displacement} shows that the graph distance traveled by the walk after $n$ steps is typically at most $n^{1/d_\gamma + o_n(1)}$, so that the walk displacement exponent is the reciporical of the ball volume exponent.

\begin{remark}
The results in this paper for uniform random planar maps are stated only for the UIPT, not for other infinite-volume uniform random planar maps such as the uniform infinite planar quadrangulation (UIPQ). The reason for this is that we do not have mating-of-trees bijections for these other random planar maps. 
We expect that it is possible to transfer our results to other uniform infinite random planar maps, including the UIPQ and more generally uniform infinite $p$-angulations for $p\geq 4$, but we do not carry this out here.  
\end{remark}

\subsection{Outline}
\label{sec-outline}
 
In Section~\ref{sec-prelim}, we introduce some basic notation which we will use throughout the rest of the paper, recall some facts about effective resistance, and provide background on $\gamma$-Liouville quantum gravity surfaces and the results from~\cite{wedges} which relate the mated-CRT map to SLE-decorated LQG. 

In Section~\ref{sec-eff-res}, we focus exclusively on the mated-CRT map $\mcl G$ under its a priori embedding into $\BB C$ which comes from SLE-decorated LQG, as explained in Section~\ref{sec-peanosphere}. Our main goal is to prove up-to-constants bounds for the effective resistance in $\mcl G$ from the origin to the boundary of a graph-distance ball, i.e., to prove the bound~\eqref{eqn-eff-res0} from Theorem~\ref{thm-green} (Proposition~\ref{prop-sg-eff-res}; c.f.\ Section~\ref{sec-eff-res-def} for a review of the definition of effective resistance).
To accomplish this, we will first prove up-to-constants bounds for the effective resistance to the boundary of a \emph{Euclidean} ball, when the map is given the a priori embedding which comes from space-filling SLE.
These bounds are proven using estimates for the Dirichlet energy of certain harmonic functions on $\mcl G $ (which build on results from~\cite{gms-harmonic}). 

In Section~\ref{sec-discrete}, we first review a coupling between any one of the first five random planar maps listed in Section~\ref{sec-main-results} and the mated-CRT map with the same parameter $\gamma$, which was originally obtained in~\cite{ghs-map-dist} using~\cite{zaitsev-kmt}. We show that under this coupling, the Dirichlet energies of functions on the mated-CRT map and the other map are comparable up to polylogarithmic factors (Lemma~\ref{lem-energy-compare}). We then use this coupling and the main estimate of Section~\ref{sec-eff-res} to prove a polylogarithmic upper bound for the effective resistance to the boundary of a metric ball in the map $M$ (Proposition~\ref{prop-map-eff-res}) and deduce our main results from this upper bound. 

Appendix~\ref{sec-return-lower} contains a proof that the return probability lower bound from~\cite{lee-conformal-growth} extends to maps with multiple loops and/or self-edges. Appendix~\ref{sec-index} contains an index of some commonly used symbols.

\section{Preliminaries}
\label{sec-prelim}

\subsection{Basic notation}  
\label{sec-basic-notation}

\noindent
We write $\BB N$ for the set of positive integers and $\BB N_0 = \BB N\cup \{0\}$. 
\vspace{6pt}

\noindent
For $a,b \in \BB R$ with $a<b$ and $r > 0$, we define the discrete intervals $[a,b]_{r \BB Z} := [a, b]\cap (r \BB Z)$ and $(a,b)_{r \BB Z} := (a,b)\cap (r\BB Z)$.
\vspace{6pt} 

\noindent
If $a$ and $b$ are two quantities we write $a\preceq b$ (resp.\ $a \succeq b$) if there is a constant $C > 0$ (independent of the values of $a$ or $b$ and certain other parameters of interest) such that $a \leq C b$ (resp.\ $a \geq C b$). We write $a \asymp b$ if $a\preceq b$ and $a \succeq b$. We typically describe dependence of implicit constants in lemma/proposition statements and require constants in the proof to satisfy the same dependencies.
\vspace{6pt}
 
\noindent
If $a$ and $b$ are two quantities depending on a variable $x$, we write $a = O_x(b)$ (resp.\ $a = o_x(b)$) if $a/b$ remains bounded (resp.\ tends to 0) as $x\to 0$ or as $x\to\infty$ (the regime we are considering will be clear from the context). We write $a = o_x^\infty(b)$ if $a = o_x(b^s)$ for every $s\in\BB R$.  
\vspace{6pt}
 
\noindent
For a graph $G$, we write $\mcl V(G)$ and $\mcl E(G)$, respectively, for the set of vertices and edges of $G$, respectively. We sometimes omit the parentheses and write $\mcl VG = \mcl V(G)$ and $\mcl EG = \mcl E(G)$. For $v\in\mcl V(G)$, we write $\op{deg}^G(v)$ for the degree of $v$ (i.e., the number of edges with $v$ as an endpoint). For vertices $v_1,v_2 \in \mcl V(G)$, we say that $v_1\sim v_2$ in $G$ if $v_1$ and $v_2$ are connected by an edge in $G$. 
\vspace{6pt}

\noindent
For $r > 0$ and $z\in\BB C$ we write $B_r(z)$ for the open disk of radius $r$ centered at $z$. 
We abbreviate $B_r = B_r(0)$. 
\vspace{6pt}

\subsection{Harmonic functions and effective resistance}
\label{sec-eff-res-def}

The main tool in the proofs of our main theorems are various estimates for discrete harmonic functions on the mated-CRT map. Here we define some of the quantities that we will study, starting with \emph{Dirichlet energy}. 
 
\begin{defn} \label{def-discrete-dirichlet}
For a graph $G$ and a function $f : \mcl V(G) \to \BB R$, we define its \emph{Dirichlet energy} to be the sum over unoriented edges
\eqbn
\op{Energy}(f; G) := \sum_{\{x,y\} \in \mcl E(G)} (f(x) - f(y))^2 ,
\eqen
with edges of multiplicity $m$ counted $m$ times.
\end{defn}

Dirichlet energy is closely related to \emph{effective resistance}, which will also be important for us. 
We view a graph $G$ as an electrical network where each edge has unit resistance. 
For a vertex $x \in \mcl V(G)$ and a set $V\subset \mcl V(G)$ with $x\notin V$, the \emph{effective resistance} from $x$ to $V$ in $G$ is defined (using Definition~\ref{def-rw-law}) by
\allb \label{eqn-eff-res-def}
\mcl R^G\left( x \leftrightarrow V \right) 
&:= \op{deg}^G\left( x  \right)^{-1} \ol{\BB E}_x^G\left[ \# \left\{\text{times $X^G$ returns to $x$ before hitting $V$} \right\} \right] \notag \\ 
&=  \op{deg}^G\left( x \right)^{-1} \op{Gr}_{\tau_V}^G(x,x)
\alle
where in the last equality $\tau_V$ is the first time that $X^G$ hits $V$ and $\op{Gr}_{\tau_V}^G(\cdot,\cdot)$ is the Green's function for random walk stopped at time $\tau_V$ (as in Section~\ref{sec-main-results}). 

There is an equivalent representation for $\mcl R^G(x\leftrightarrow V)$ in terms of Dirichlet energy. Namely, let $\frk f_V  : \mcl V(G)  \to [0,1]$ be the function such that $\frk f_V(x ) = 1$, $\frk f_V|_V \equiv 0$, and $\frk f_V $ is discrete harmonic elsewhere. Then by \emph{Dirichlet's principle} (see, e.g.,~\cite[Exercise 2.13]{lyons-peres}),
\eqb \label{eqn-eff-res-dirichlet}
\mcl R^G\left( x \leftrightarrow V \right)  = \frac{1}{\op{Energy} (\frk f_V  ;  G ) } .
\eqe 

We will also need a third equivalent representation for effective resistance in terms of so-called unit flows, which we will use in Section~\ref{sec-eff-res-upper}. A \emph{unit flow} from $x$ to $V$ in $G$ is a function $\theta$ from oriented edges $e = (y,z)$ of $G$ to $\BB R$ such that $\theta(y,z) = -\theta(z,y)$ for each oriented edge $(y,z)$ of $G$ and
\eqbn 
\sum_{\substack{z \in \mcl V(G) \\ z\sim y}} \theta(y,z) = 0 \quad \forall z \in \mcl V(G)\setminus  ( \{x\} \cup V )   \quad\op{and} \quad
\sum_{\substack{z \in \mcl V(G) \\ z\sim x}} \theta(x,z) =1 .
\eqen
The quantity $\sum_{ z \in \mcl V(G) :  z\sim y } \theta(y,z)$ is called the \emph{divergence} of $\theta$ at $y$. 
By \emph{Thomson's principle} (see, e.g.,~\cite[Theorem 9.10]{markov-mixing}),
\eqb \label{eqn-thomson}
\mcl R^G(x\leftrightarrow V) = \inf\left\{ \sum_{e\in\mcl E(G)} [\theta(e)]^2 : \text{$\theta$ is a unit flow from $x$ to $V$} \right\} .
\eqe
The sum appearing on the right in~\eqref{eqn-thomson} is called the \emph{energy} of the flow $\theta$, by analogy with Definition~\ref{def-discrete-dirichlet}. We note that the sum is over unoriented edges, and that $[\theta(e)]^2$ is well-defined for an unoriented edge $e$ due to the anti-symmetry condition above.

\subsection{Liouville quantum gravity and the $\gamma$-quantum cone}  
\label{sec-lqg-prelim}

Heuristically speaking, a $\gamma$-Liouville quantum gravity (LQG) surface for $\gamma \in (0,2)$ is the random Riemannian surface parameterized by a domain $D\subset \BB C$ with Riemannian metric tensor $e^{\gamma h(z)} \, (dx^2+dy^2)$, where $h$ is some variant of the Gaussian free field (GFF) on $D$. We assume that the reader is familiar with the Gaussian free field; see~\cite{shef-gff,ss-contour,ig1,ig4} for more details. Of course, the preceding definition of a $\gamma$-LQG surface does not make rigorous sense since $h$ is a distribution, not a function, so cannot be exponentiated. Nevertheless, one can make rigorous sense of $\gamma$-LQG in various ways. 

Duplantier and Sheffield~\cite{shef-kpz} rigorously constructed the volume form associated with a $\gamma$-LQG surface, a measure $\mu_h$ which is the limit of regularized versions of $e^{\gamma h(z)} \,dz$, where $dz$ denotes Lebesgue measure. One can similarly define a $\gamma$-LQG boundary length measure $\nu_h$ on certain curves in $D$, including $\bdy D$ and SLE$_\kappa$-type curves for $\kappa = \gamma^2$~\cite{shef-zipper}. See~\cite{rhodes-vargas-review} for a review of a more general theory of regularized measures of this form, which dates back to Kahane~\cite{kahane}. 

Hence it makes sense to think of an LQG surface as a random measure space with a conformal structure. One would like to allow for different parameterizations of the same surface, so we consider equivalence classes.  For $\gamma \in (0,2)$ and $k\in\BB N_0$, a \emph{$\gamma$-LQG surface} with $k$ marked points is an equivalence class of $k+2$-tuples $(D ,h , z_1,\dots,z_k)$ where $D\subset \BB C$, $h$ is a distribution on $D$ (typically some variant of the GFF), and $z_1,\dots,z_k$ are marked points in $D\cup \bdy D$. Two such $k+2$-tuples $(D ,h , z_1,\dots,z_k)$ and $(\wt D , \wt h , \wt z_1 , \dots , \wt z_k)$ are declared to be equivalent (heuristically, this means they represent different parameterizations of the same surface) if there is a conformal map $f : \wt D \to D$ such that 
\eqb \label{eqn-lqg-coord}
\wt h = h \circ f + Q\log |f'| \quad \op{for} \quad Q = \frac{2}{\gamma} + \frac{\gamma}{2} \quad\op{and} \quad f(\wt z_j) = z_j ,\quad \forall j\in [1,k]_{\BB Z} .
\eqe
A particular choice of the distribution $h$ is called an \emph{embedding} of the $\gamma$-LQG surface. The reason for this definition is that if $h$ and $\wt h$ are related as in~\eqref{eqn-lqg-coord}, then the $\gamma$-LQG measures a.s.\ satisfy $\mu_h = f_* \mu_{\wt h}$ and $\nu_h = f_* \nu_{\wt h}$~\cite[Proposition~2.1]{shef-kpz}. 

The only type of $\gamma$-LQG surface in which we will be interested in this paper is the \emph{$\gamma$-quantum cone}, which was first defined in~\cite[Definition~4.10]{wedges}.  The $\gamma$-quantum cone is an infinite-volume LQG surface (i.e., the $\gamma$-LQG measure $\mu_h$ has infinite total mass) with two marked points, parameterized by $\BB C$.  A $\gamma$-quantum cone can be represented by $(\BB C , h , 0, \infty)$ for a certain type of distribution $h$ on $\BB C$, which is a slight modification of a whole-plane GFF plus $-\gamma \log |\cdot|$. 

Let $A : \BB R \to \BB R$ be the process such that $A_t =B_t  + \gamma t$ for $t\geq 0$, where $B$ is a standard linear Brownian motion; and for $t < 0$, let $A_t = \wh B_{-t} + \gamma t$, where $\wh B$ is a standard linear Brownian motion conditioned so that $\wh B_t  + (Q-\gamma) t > 0$ for all $t> 0$ (this singular conditioning is made sense of in~\cite[Remark 4.4]{wedges}).  We define $h$ to be the random distribution such that if $h_r(0)$ denotes the circle average of $h$ on $\partial B_r(0)$ (see~\cite[Section~3.1]{shef-kpz} for the definition and basic properties of the circle average), then $t\mapsto h_{e^{-t}}(0)$ has the same law as the process $A$; and $h - h_{|\cdot|}(0)$ is independent from $h_{|\cdot|}(0)$ and has the same law as the analogous process for a whole-plane GFF. 
 
By the definition of an LQG surface, the distribution $h$ is only defined up to re-scaling (as we have fixed only two marked points), but we will almost always consider the particular choice of embedding $h$ defined just above.  This choice of $h$ is called the \emph{circle average embedding}.  The circle average embedding possesses two key properties which are essentially immediate from the definition and will be important for our purposes.  The first property is that $h|_{\BB D}$ agrees in law with the corresponding restriction of a whole-plane GFF plus $-\gamma \log |\cdot|$, normalized so that its circle average over $\bdy \BB D$ is~$0$. 

The other property we will need is a certain scale invariance, which we now describe.
For $r > 0$ and $z\in \BB C$, let $h_r(z)$ be the circle average of $h$ over $\bdy B_r(z)$ and for $b > 0$, let
\eqb \label{eqn-mass-hit-time}
R_b := \sup\left\{ r > 0 : h_r(0) + Q \log r = \frac{1}{\gamma} \log b \right\} 
\eqe 
where here $Q$ is as in~\eqref{eqn-lqg-coord}. That is, $R_b$ gives the largest radius $r > 0$ so that if we scale spatially by the factor $r$ and apply the change of coordinates formula~\eqref{eqn-lqg-coord}, then the average of the resulting field on $\partial \BB D$ is equal to $\gamma^{-1} \log b$. 
Note that $R_0 = 1$ by the definition of the circle average embedding. It is easy to see from the definition of $h$ (and is shown in~\cite[Proposition~4.13(i)]{wedges}) that for each fixed $b>0$, 
\eqb \label{eqn-cone-scale}
h \eqD h(R_b \cdot) + Q \log R_b -  \frac{1}{\gamma} \log b .
\eqe
By~\eqref{eqn-lqg-coord}, if we let $h^b$ be the field on the right side of~\eqref{eqn-cone-scale}, then a.s.\ $\mu_{h^b}(A) = b \mu_h(R_b^{-1} A)$ for each Borel set $A\subset \BB C$. In particular, typically $\mu_h(B_{R_b}) \asymp b$. We will use the scale invariance property~\eqref{eqn-cone-scale} in Section~\ref{sec-eff-res} to transfer estimates at macroscopic scales to estimates at microscopic scales.
We will also need the following basic estimate for the radii $R_b$ in~\eqref{eqn-mass-hit-time}.

\begin{lem} \label{lem-cone-hit-tail}
Suppose $h$ is the circle average embedding of a $\gamma$-quantum cone. There is a constant $a = a(\gamma)  > 0$ such that for each $b_2>b_1>0$ and each $C > 1$,
\begin{align} \label{eqn-cone-hit-tail}
&\BB P\left[ C^{-1} (b_2/b_1)^{ \tfrac{1 }{\gamma (Q-\gamma) }} \leq R_{b_2}/R_{b_1} \leq C (b_2/b_1)^{ \tfrac{1 }{\gamma (Q-\gamma)}} \right] \notag\\
&\qquad\qquad\qquad \geq 1  -   3 \exp\left( - \frac{  a  (\log C)^2 }{   \log (b_2/b_1) +  \log C   } \right)  .
\end{align}
\end{lem}
\begin{proof}
This is a re-statement of~\cite[Lemma~2.1]{gms-tutte} in the special case when $\alpha=\gamma$.
\end{proof}

\subsection{Mated-CRT maps and SLE-decorated Liouville quantum gravity}  
\label{sec-peanosphere}

In this subsection we will describe the connection between mated-CRT maps and SLE-decorated LQG, as alluded to at the end of Section~\ref{sec-mated-crt-map}.  This connection will be our primary tool for analyzing mated-CRT maps. 

\emph{Schramm-Loewner evolution} (SLE$_\kappa$) for $\kappa > 0$ is a family of random fractal curves in the plane first introduced by Schramm~\cite{schramm0}. In this paper, we will be interested in a variant of SLE$_\kappa$ for $\kappa > 4$ called \emph{whole-plane space-filling SLE$_\kappa$ from $\infty$ to $\infty$} which is introduced in~\cite[Sections~1.2.3 and~4.3]{ig4} (see also~\cite[Section 1.4.1]{wedges} for the whole-plane case). This is a continuous, space-filling, non-crossing curve in $\BB C$ which a.s.\ hits Lebesgue-a.e.\ point of $\BB C$ exactly once. 

For $\kappa \geq 8$, ordinary SLE$_\kappa$ is space-filling and whole-plane space-filling SLE$_\kappa$ from $\infty$ to $\infty$ is a two-sided variant of chordal SLE$_\kappa$. In the case when $\kappa \in (4,8)$, ordinary SLE$_\kappa$ is not space-filling, and instead hits itself to form ``bubbles" which it disconnected from its target point. In this case, space-filling SLE$_\kappa$ is obtained from ordinary SLE$_\kappa$ by iteratively filling in these bubbles with SLE$_\kappa$-type curves (so in particular it is not a Loewner evolution). See~\cite[Section 3.6.3]{ghs-mating-survey} for a precise description of this construction. 

Space-filling SLE is a.s.\ a continuous curve when parametrized so that it traverses one unit of Lebesgue measure in one unit of time. The same is true with the LQG measure corresponding to, e.g., a $\gamma$-quantum cone used in place of the Lebesgue measure. Under either of these parametrizations, for almost every time $t$, the curve $\eta((-\infty,t])$ has well-defined left and right boundaries, which roughly speaking correspond to the sets of points on $\bdy\eta((-\infty,t])$ which lie to the left and right of the tip $\eta(t)$, respectively. The left and right boundaries of $\eta((-\infty,t])$ are each SLE$_{16/\kappa}$ type curves which can be realized as two flow lines of a whole-plane Gaussian free field, in the sense of~\cite{ig4}; see~\cite[Section 1.4.1]{wedges}. We will not need any further details concerning SLE$_\kappa$ and its variants in this paper.

Mated-CRT maps are related to SLE-decorated LQG via the peanosphere (or mating-of-trees) construction of~\cite[Theorem~1.9]{wedges}, which we now describe. See also the right side of Figure~\ref{fig-bijections}.  Suppose $h$ is the circle-average embedding of a $\gamma$-quantum cone, as in Section~\ref{sec-lqg-prelim}.  Also let $\kappa = 16/\gamma^2  > 4$ and let $\eta$ be a whole-plane space-filling SLE$_\kappa$ from $\infty$ to $\infty$ sampled independently from $h$ and then parameterized in such a way that $\eta(0) = 0$ and the $\gamma$-LQG mass satisfies $\mu_h(\eta([t_1,t_2])) = t_2 - t_1$ whenever $t_1 , t_2 \in \BB R$ with $t_1 < t_2$.

Let $\nu_h$ be the $\gamma$-LQG length measure associated with $h$ and define a process $L : \BB R \to \BB R$ in such way that $L_0 = 0$ and for $t_1,t_2 \in \BB R$ with $t_1<t_2$,
\allb \label{eqn-peanosphere-bm}
L_{t_2}- L_{t_1} &= \nu_h\left( \text{left boundary of $\eta([t_1,t_2]) \cap \eta([t_2,\infty))$} \right) \notag \\
&\qquad - \nu_h\left( \text{left boundary of $\eta([t_1,t_2]) \cap \eta((-\infty , t_1])$} \right).
\alle 
Define $R_t$ similarly but with ``right" in place of ``left" and set $Z_t = (L_t , R_t)$. It is shown in~\cite[Theorem~1.9]{wedges} that $Z$ evolves as a correlated two-dimensional Brownian with variances and covariances as in~\eqref{eqn-bm-cov}, so $Z$ has the same law as the Brownian motion used to construct the mated-CRT map with parameter $\gamma$. Moreover, by~\cite[Theorem 1.11]{wedges}, $Z$ a.s.\ determines $(h,\eta)$ modulo rotation.

We can re-phrase the adjacency condition~\eqref{eqn-inf-adjacency} in terms of $(h,\eta)$. In particular, it follows from from~\eqref{eqn-peanosphere-bm} (see~\cite[Section 8.2]{wedges}) that for $x_1,x_2 \in  \BB Z$ with $x_1 < x_2$,~\eqref{eqn-inf-adjacency} is satisfied if and only if the ``cells" $\eta([x_1- 1 , x_1])$ and $\eta([x_2- 1 , x_2])$ intersect along a non-trivial connected arc of their left outer boundaries; and similarly with ``$R$" in place of ``$L$" and ``left" in place of ``right". 

Consequently, the mated-CRT map $\mcl G$ is precisely the graph with vertex set $ \BB Z$, with two vertices connected by an edge if and only if the corresponding cells $\eta([x_1- 1 , x_1])$ and $\eta([x_2- 1 , x_2])$ share a non-trivial connected boundary arc (the vertices are connected by two edges if $|x_1-x_2| > 1$ and the cells intersect along both their left and right boundaries). 
The graph on cells is sometimes called the \emph{structure graph} of the curve $\eta$ since it encodes the topological structure of the cells. 
The identification of $\mcl G $ with the structure graph of $\eta$ gives us an a priori embedding of $\mcl G $ into $\BB C$ by sending each vertex $x\in \BB Z$ to the point $\eta(x)$.

\section{Effective resistance on the mated-CRT map}
\label{sec-eff-res}

Fix $\gamma \in (0,2)$ and let $\mcl G$ be the $\gamma$-mated-CRT map, as in Section~\ref{sec-mated-crt-map}.  Recall the definition of effective resistance from~\eqref{eqn-eff-res-def}. The goal of this section is to prove the following bound for the effective resistance in $\mcl G$ from the origin (i.e., the root vertex) to the boundary of a metric ball, which is a restatement of~\eqref{eqn-eff-res0} of Theorem~\ref{thm-green}.
 
\begin{prop} \label{prop-sg-eff-res}
There exists $\alpha=\alpha(\gamma)  > 0$ and $C = C(\gamma)  >0$ such that for each $r \in\BB N$, it holds with probability at least $1-O_r( (\log r)^{-\alpha})$ that
\eqb \label{eqn-sg-eff-res}
C^{-1} \log r \leq   \mcl R^{\mcl G} \left(  0 \leftrightarrow \bdy \mcl B_r^{\mcl G}(0)      \right)   \leq   C  \log r ,
\eqe 
where $\bdy \mcl B_r^{\mcl G}(0)$ denotes the set of vertices of $\mcl B_r^{\mcl G}(0)$ which are adjacent to vertices not in $\mcl B_r^{\mcl G}(0)$.
\end{prop} 
 
Proposition~\ref{prop-sg-eff-res} is the only result from this section which is needed in Section~\ref{sec-discrete}. 
The proof of Proposition~\ref{prop-sg-eff-res} uses the relationship between $\mcl G$ and SLE-decorated LQG, as explained in Section~\ref{sec-peanosphere}, together with the bounds for harmonic functions on $\mcl G$ from~\cite{gms-harmonic}.

\subsection{Setup and outline}
\label{sec-eff-res-setup}

Throughout this section we will consider the following setup.
Fix $\gamma \in (0,2)$ and
let $Z = (L,R)$ be the correlated Brownian motion as in~\eqref{eqn-bm-cov} used to define the mated-CRT map $\mcl G$ with this choice of $\gamma$. 
It will be convenient to consider a collection $\{\mcl G^\ep\}_{\ep > 0}$ of graphs with the same law as $\mcl G$, all coupled with $Z$, defined as follows. The vertex set of $\mcl G^\ep$ is $\ep\BB Z$, and two vertices $x_1,x_2\in  \ep\BB Z$ with $x_1<x_2$ connected by an edge if and only if
\eqb  \label{eqn-inf-adjacency-ep}
\left( \inf_{t\in [x_1- \ep , x_1]} L_t \right) \vee \left( \inf_{t\in [x_2- \ep , x_2]} L_t \right) \leq \inf_{t\in [x_1  , x_2 - \ep]} L_t 
\eqe 
or the same holds with $R$ in place of $L$. In other words, $\mcl G^\ep$ is defined in the same manner as $\mcl G$ but with the Brownian motion $t\mapsto \ep^{ 1/2} Z_{ t / \ep}$ in place of $Z$. Note that $\mcl G = \mcl G^1$ and (by Brownian scaling) for every $\ep > 0$, $\mcl G^\ep$ agrees in law with $ \mcl G$ viewed as a graph with a total ordering on its vertices.
 
Let $(h,\eta)$ be the pair consisting of the circle-average embedding of a $\gamma$-quantum cone and an independent whole-plane space-filling SLE$_{\kappa}$ with $\kappa =16/\gamma^2$ which is determined by $Z$ via~\cite[Theorem 1.11]{wedges}, as explained in Section~\ref{sec-peanosphere}. 
Then two vertices $x_1,x_2\in\mcl V\mcl G^\ep = \ep\BB Z$ are connected by an edge if and only if the corresponding cells $\eta([x_1-\ep ,x_1])$ and $\eta([x_2-\ep , x_2])$ share a non-trivial boundary arc.  
 
For a set $D\subset \BB C$, we define $\mcl G^\ep(D)$ to be the sub-graph of $\mcl G^\ep$ with vertex set
\eqb \label{eqn-sg-domain}
\mcl V\mcl G^\ep(D) := \left\{ x\in \ep\BB Z  : \eta([x-\ep , x]) \cap D\not=\emptyset \right\}   
\eqe
with two vertices joined by an edge if and only if they are joined by an edge in $\mcl G^\ep$. See Figure~\ref{fig-structure-graph-restrict} for an illustration.

\begin{figure}[ht!]
 \begin{center}
\includegraphics[scale=.75]{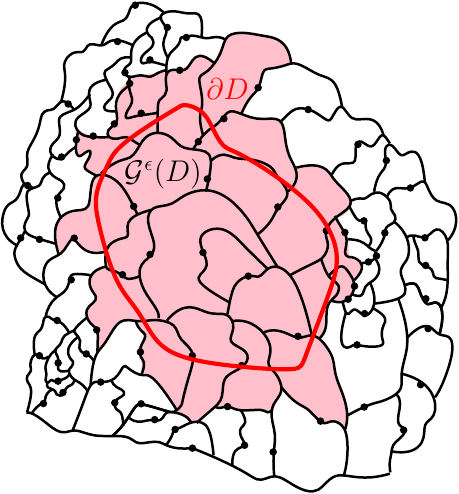} 
\caption{For a set $D\subset \BB C$ (here shown as the region enclosed by the red curve), the sub-graph $\mcl G^\ep(D)$ is the graph of cells which intersect $D$ (light red), with two such cells connected by an edge if and only if they share a non-trivial boundary arc. 
}\label{fig-structure-graph-restrict}
\end{center}
\end{figure}
 
We will make frequent use of the following upper bound for the maximal Euclidean diameter of the cells of $\mcl G^\ep$ which intersect a Euclidean ball of fixed radius, which is~\cite[Lemma~2.4]{gms-harmonic}.

\begin{lem}[\cite{gms-harmonic}] \label{lem-max-cell-diam}
For each $q\in \left( 0 , \tfrac{2}{(2+\gamma)^2} \right)$, there exists $\alpha=\alpha(q ,\gamma ) > 0$ such that for each fixed $\rho\in (0,1)$,  
\eqbn
\BB P\left[ \max_{x\in\mcl V\mcl G^\ep(B_\rho)} \op{diam}\left( \eta([ x-\ep , x]) \right) \leq \ep^q \right] = 1 - O_\ep(\ep^\alpha)  \quad \op{as} \quad \ep \to 0.
\eqen
\end{lem} 

To prove Proposition~\ref{prop-sg-eff-res}, we start by considering a fixed radius $\rho \in (0,1)$ and proving up-to-constants lower and upper bounds for the effective resistance $\mcl R^{\mcl G^\ep}(0 \leftrightarrow \mcl V\mcl G^\ep(\bdy B_\rho))$ from $0 \in \mcl V\mcl G^\ep$ to the set $\mcl V\mcl G^\ep(\bdy B_\rho)$ of vertices corresponding to the cells which intersect the boundary of the Euclidean ball of radius $\rho$ centered at 0.  

The lower bound for effective resistance to $\bdy B_\rho$ (Proposition~\ref{prop-eff-res-lower}) is just a re-statement of~\cite[Theorem 1.4]{gms-harmonic}. 
The proof of the corresponding upper bound (Proposition~\ref{prop-eff-res-upper}) in Section~\ref{sec-eff-res-upper} (which is more important for our purposes) will be proven by using~\eqref{eqn-thomson} and a multi-scale argument to get an upper bound for the energy of a certain unit flow (this argument is outlined just after the proposition statement).

In Section~\ref{sec-ball-eff-res}, we deduce Proposition~\ref{prop-eff-res-lower} from our estimates for the effective resistance to the boundary of a Euclidean ball using that $\mcl G^\ep \eqD \mcl G$ (as graphs with an ordering on the vertices) and that with high probability 
\eqbn
\mcl V\mcl B_{\ep^{-q_1}}^{\mcl G^\ep}(0)  \subset \mcl V\mcl G^\ep(B_\rho) \subset \mcl V\mcl B_{\ep^{-q_2}}^{\mcl G^\ep}(0)   
\eqen
for constants $q_1,q_2 > 0$ depending only on $\gamma$.

Throughout this section, for $0  < a < b < \infty$ and $z\in\BB C$, we define the open annulus
\eqb \label{eqn-annulus-def}
\BB A_{a,b}(z) := B_b(z) \setminus \ol{B_a(z)} .
\eqe
We also declare that $\BB A_{0,b}(z)$ is the punctured disk $B_b(z) \setminus \{z\}$.
We abbreviate $\BB A_{a,b} := \BB A_{a,b}(0)$.

\subsection{Lower bound for effective resistance to $\bdy B_\rho$}
\label{sec-eff-res-lower}

We record for reference the following lower bound for the effective resistance in $\mcl G^\ep$ from $0$ to the boundary of the Euclidean ball $B_\rho$, which is proven in~\cite{gms-harmonic}.

\begin{prop} \label{prop-eff-res-lower} 
There exists $\alpha = \alpha(\gamma) > 0$ such that for each $\rho \in (0,1)$, there exists $C = C( \rho,\gamma) > 0$ such that for $\ep\in (0,1)$,  
\eqb \label{eqn-eff-res-lower} 
\BB P\left[   \mcl R^{\mcl G^\ep}\left(0 \leftrightarrow \mcl V\mcl G^\ep( \bdy B_\rho )  \right)  \geq \frac{1}{C} \log \ep^{-1}  \right] = 1 - O_\ep(\ep^\alpha)
\eqe 
at a rate depending only on $\rho$ and $\gamma$.
\end{prop}
\begin{proof}
This follows from~\cite[Theorem 1.4]{gms-harmonic} and the definition~\eqref{eqn-eff-res-def} of effective resistance.
\end{proof}

\subsection{Upper bound for effective resistance to $\bdy B_\rho$}
\label{sec-eff-res-upper}

To complement Proposition~\ref{prop-eff-res-upper}, in this subsection we prove the following upper bound for the effective resistance from $0$ to $\bdy B_\rho$ in $\mcl G^\ep$.

\begin{prop} \label{prop-eff-res-upper}
There exists $\alpha=\alpha(  \gamma) > 0$ and $C = C(\rho,\gamma) > 0$ such that for $\ep\in (0,1)$, the effective resistance from $0$ to $\mcl V\mcl G^\ep(\bdy B_\rho)$ satisfies
\eqb \label{eqn-eff-res-upper}
\BB P\left[ \mcl R^{\mcl G^\ep}\left( 0 \leftrightarrow \mcl V(\bdy B_\rho )   \right) \leq  C \log \ep^{-1} \right] = 1 - O_\ep\left( \frac{1}{(\log \ep^{-1})^\alpha} \right) .
\eqe  
\end{prop}

We emphasize that the probabilistic estimate in Proposition~\ref{prop-eff-res-upper} is $1-O_\ep( (\log \ep^{-1})^{-\alpha})$ in contrast to the $1- O_\ep(\ep^\alpha)$ in Proposition~\ref{prop-eff-res-lower}. This estimate is the source of all of the polylogarithmic bounds for unconditional probabilities in this paper. 
 
Recalling Thomson's principle~\eqref{eqn-thomson}, we seek a unit flow from 0 to $\mcl V\mcl G^\ep(\bdy B_\rho)$ in $\mcl V\mcl G^\ep(B_\rho)$ whose energy is at most a constant times $\log \ep^{-1}$. 
Let us now define the unit flow we will consider.  Let $\BB z$ be sampled uniformly from Lebesgue measure on $\bdy B_\rho$, independently from everything else, and let $S$ be the line segment from 0 to $\BB z$. For $\ep\in (0,1)$, choose (in some measurable manner) a simple path $P^\ep$ in $\mcl G^\ep(S)$ from $0$ to a vertex whose corresponding cell contains $\BB z$. For an oriented edge $e  $ of $\mcl G^\ep(B_\rho)$, let $\theta^\ep(e)$ be the probability that the path $P^\ep$ traverses $e$ in the forward direction, minus the probability that $P^\ep$ traverses $e$ in the reverse direction. 

\begin{lem}
The function $\theta^\ep$ defined just above is a unit flow from $0$ to $\mcl V\mcl G^\ep(\bdy B_\rho)$. 
\end{lem}
\begin{proof}
For each $x\in \mcl V\mcl G^\ep(B_\rho) \setminus \mcl V\mcl G^\ep(\bdy B_\rho)$, the number of oriented edges of the form $(x,y)$ traversed by $P^\ep$ is equal to the number of oriented edges of the form $(y,x)$ traversed by $P^\ep$, i.e., the net number of times that $P^\ep$ traverses an oriented edge with $x$ as an endpoint is 0, where edges pointing toward $x$ are counted with a positive sign and edges pointed away from $x$ are counted with a negative sign. The divergence of $\theta^\ep$ at $x$ is the conditional expectation of this net number of edges given $(h,\eta)$, so also equals $0$. On the other hand, the net number of times that the path $P^\ep$ traverses an edge with~$0$ as an endpoint (defined as above) is $1$. Therefore, the divergence at $0$ is $1$, so $\theta^\ep$ is in fact a unit flow.
\end{proof}

To prove Proposition~\ref{prop-eff-res-upper}, it suffices to prove an upper bound of $O_\ep(\log\ep^{-1})$ for the energy of the unit flow $\theta^\ep$. To do this, we first use the estimates for cells of $\mcl G^\ep$ from~\cite{gms-harmonic} to prove an upper bound for the energy of $\theta^\ep$ over the annulus $\BB A_{\ep^\beta , \rho}$ for a small (but $\ep$-independent) constant $\beta=\beta(\gamma) >0$ (Lemma~\ref{lem-flow-path-macro}). Intuitively, the reason why we do not get estimates for the energy within $B_{\ep^\beta}$ is that there are too few cells contained in $B_{\ep^\beta}$ for various large-scale averaging effects to hold. (This manifests itself, e.g., in the fact that the upper bound for the cell size in Lemma~\ref{lem-max-cell-diam} might be bigger than the size of the ball we are working with and the fact that the $\ep^\alpha$ error term in~\cite[Lemma 3.1]{gms-harmonic} dominates the main integral term when we are working in a small enough ball.)
 
To transfer from this energy bound at ``nearly macroscopic" scales to an energy bound which works at all scales, we will apply the scaling property of the $\gamma$-quantum cone described in Section~\ref{sec-lqg-prelim}. 
Via a multi-scale argument, this leads to an upper bound of $O_\ep(\log \ep^{-1})$ for the energy of $\theta^\ep$ over the complement of a Euclidean ball centered at 0 which contains of order $T > 1$ cells of $\mcl G^\ep$, which holds with probability $1-O_T(T^{-\alpha})$ (Lemma~\ref{lem-flow-path-micro}). 

We have an a priori bound for the energy of $\theta^\ep$ over a ball which contains $O_\ep( (\log \ep^{-1})^{1/4} )$ cells of $\mcl G^\ep$ (Lemma~\ref{lem-flow-path-tiny}), which allows us to take $T = (\log \ep^{-1})^{1/4}$ in the preceding estimate and thereby conclude the proof of Proposition~\ref{prop-eff-res-upper}. Note that the need to consider balls which contain a polylogarithmic number of cells of $\mcl G^\ep$ is the reason for the polylogarithmic probability bound in Proposition~\ref{prop-eff-res-upper}.

\begin{lem} \label{lem-flow-path-macro}
There exists $\alpha=\alpha(\gamma) >0$, $\beta = \beta(\gamma) > 0$, and $C=C(\rho,\gamma) > 0$ such that the energy of $\theta^\ep$ over $\mcl G^\ep(\BB A_{\ep^\beta,\rho})$ satisfies
\eqb \label{eqn-flow-path-macro} 
\BB P\left[ \sum_{e \in \mcl E\mcl G^\ep(\BB A_{\ep^\beta,\rho})} [\theta^\ep(e)]^2 \leq C \log \ep^{-1} \right] =  1-O_\ep(\ep^\alpha) .
\eqe 
\end{lem}
\begin{proof}
We will bound the energy of $\theta^\ep$ in terms of a sum over cells of $\mcl G^\ep$ which intersect $\BB A_{\ep^\beta,\rho}$, which can in turn be bounded using~\cite[Lemma~3.1]{gms-harmonic}.
Throughout the proof, we require all implicit constants in the symbol $\preceq$ to be deterministic and depend only on $\rho$ and $\gamma$.
Fix $q\in \left(0,\tfrac{2}{(2+\gamma)^2}\right)$, chosen in a manner depending only on $\gamma$ (as in Lemma~\ref{lem-max-cell-diam}). 

For each $x\in \mcl V\mcl G^\ep(B_\rho)$, the conditional probability given $(h,\eta)$ that the segment $S$ intersects the cell $\eta([x-\ep,x])$ is at most a constant times\\
 $\op{diam} \left(\eta([x-\ep,x]) \right) \times \op{dist}\left(\eta([x-\ep,x]) , 0\right)^{-1}$, where here $\op{diam}$ and $\op{dist}$ denote Euclidean diameter and distance, respectively. By Lemma~\ref{lem-max-cell-diam}, it holds except on an event of probability decaying faster than some positive power of~$\ep$ that each cell $\eta([x-\ep,x])$ has diameter at most $\ep^q$, so
\eqbn
\op{dist}\left(\eta([x-\ep,x]) , 0\right)^{-1} \leq \left( |\eta(x)| -  \ep^q \right)^{-1} \preceq    |\eta(x)|^{-1}  ,\quad \forall x \in \mcl V\mcl G^\ep(\BB A_{2\ep^q,\rho})  .
\eqen 
Consequently, it holds except on an event of probability decaying faster than some positive power of~$\ep$ that 
\eqb \label{eqn-flow-path-prob}
\BB P\left[x \in P^\ep \,|\, (h,\eta) \right] \preceq |\eta(x)|^{-1} \op{diam} \left(\eta([x-\ep,x]) \right)      , \quad \forall x \in \mcl V\mcl G^\ep(\BB A_{2\ep^q,\rho}) .
\eqe 
By the definition of $\theta^\ep$,~\eqref{eqn-flow-path-prob} implies that for each oriented edge $e = (x,y)$ of $\mcl G^\ep(\BB A_{2\ep^q,\rho})$, 
\eqbn
|\theta^\ep(e)| \preceq |\eta(x)|^{-1} \op{diam} \left(\eta([x-\ep,x]) \right) .
\eqen
Summing this estimate, we get that if $\beta \in (0,q)$, then whenever~\eqref{eqn-flow-path-prob} holds and $\ep$ is sufficiently small (depending on $\beta$ and $q$),
\allb \label{eqn-flow-path-sum}
\sum_{e \in \mcl E\mcl G^\ep(\BB A_{\ep^\beta,\rho})}[ \theta^\ep(e)]^2
 \preceq \sum_{x \in \mcl V\mcl G^\ep(\BB A_{\ep^\beta,\rho})}  |\eta(x)|^{-2}   \op{diam} \left(\eta([x-\ep,x]) \right)^2 \op{deg}^{\mcl G^\ep}(x) .
\alle

By~\cite[Lemma~3.1]{gms-harmonic}, applied with $f(z) = |z|^{-2}$, it follows that there exists constants $\alpha_0 = \alpha_0(\gamma) > 0$, $\beta_0 = \beta_0(\gamma)  > 0$, and $C_0 =C_0(\rho,\gamma) > 0$ such that with probability $1-O_\ep(\ep^{\alpha_0})$, the right side of~\eqref{eqn-flow-path-sum} is bounded above by
\eqbn
C_0 \int_{\BB A_{\ep^\beta,\rho}} |z|^{-2} \, dz    + \ep^\alpha \preceq \log \ep^{-1} .
\eqen
Since~\eqref{eqn-flow-path-sum} holds except on an event of probability decaying faster than some positive power of $\ep$, we obtain~\eqref{eqn-flow-path-macro} for an appropriate choice of $\alpha$, $\beta$, and $C$.
\end{proof}

We now use a multi-scale argument to transfer from Lemma~\ref{lem-flow-path-macro} to a bound on the energy of $\theta^\ep$ over all of $B_\rho$ except for a small ball centered at the origin which contains of order $T $ cells of $\mcl G^\ep$ for $T \in [1,1/\ep)$. 
For the statement and proof of the next lemma, we recall the radii $R_b \in (0,1]$ for $b\in (0,1]$ defined in~\eqref{eqn-mass-hit-time}, which are chosen so that $B_{R_b} $ typically has $\mu_h$-mass of order $b$.

\begin{lem} \label{lem-flow-path-micro}
There exists $\alpha=\alpha(\gamma) >0$ such that for each $\rho \in (0,1)$, we can find $C = C( \rho , \gamma ) > 0$ such that for each $\ep \in (0,1)$ and each $T\in [1,1/\ep)$, the following is true. If we let $R_{T\ep}$ be as in~\eqref{eqn-mass-hit-time} with $b=T\ep$, then  
\eqb \label{eqn-flow-path-micro}
\BB P\left[ \sum_{e \in \mcl E\mcl G^\ep(\BB A_{R_{T\ep} ,\rho})} [\theta^\ep(e)]^2 \leq C \log \ep^{-1} \right] = 1-O_T(T^{-\alpha}) ,
\eqe 
at a rate depending only on $\rho$ and $\gamma$. 
\end{lem}
\begin{proof}
We will apply Lemma~\ref{lem-flow-path-macro} with $h(R_b\cdot) + Q\log R_b - \frac{1}{\gamma} \log b$ in place of $h$ and the scaling property~\eqref{eqn-cone-scale} of the $\gamma$-quantum cone for several particular $b \in (0,1]$ which interpolate between $1$ and $T\ep$, then take a union bound. The particular values of $b$ which we consider are chosen in~\eqref{eqn-annulus-path-b-choice} and are defined in such a way that with high probability, the corresponding intervals $[(\ep/b)^\beta R_b , R_b \rho]$ cover $[T\ep , \rho]$.  
To deduce~\eqref{eqn-flow-path-micro}, we then sum the estimate of Lemma~\ref{lem-flow-path-macro} over all of the scales.
\medskip

\noindent\textit{Step 1: Definition of an event at each scale.} Fix $\rho\in (0,1  )$ and let $\alpha_0 , \beta$, and $C_0$ be chosen so that the conclusion of Lemma~\ref{lem-flow-path-macro} is satisfied with $\alpha =\alpha_0$, $C = C_0$, and this choice of $\beta$. We can take $\beta < (100\gamma (Q-\gamma))^{-1}$. 

By the scaling property of the $\gamma$-quantum cone~\eqref{eqn-cone-scale}, the scale invariance of the law of $\eta$, and the fact that $h$ and $\eta$ are independent, for each $b>0$ we have
\eqb \label{eqn-field-curve-scale}
\left( h(R_b \cdot) + Q\log R_b - \frac{1}{\gamma} \log b , R_b^{-1} \eta(b \cdot)  \right) \eqD \left( h , \eta \right) . 
\eqe
By the LQG coordinate change formula~\eqref{eqn-lqg-coord}, a.s.\ for each $b>0$ the LQG area measure satisfies
\eqbn
\mu_{h(R_b \cdot) + Q\log R_b - \frac{1}{\gamma} \log b  } = b^{-1} \mu_h(R_b \cdot) .
\eqen
Therefore, the adjacency graph of cells $R_b^{-1} \mcl G^\ep$ is obtained from the left field / curve pair in~\eqref{eqn-field-curve-scale} in the same way that $\mcl G^{\ep/b}$ is obtained from $(h,\eta)$. Hence the collection of cells of $R_b^{-1} \mcl G^{\ep } $ has the same law as the collection of cells of $ \mcl G^{\ep/b}$. 

For $\ep\in (0,1)$ and $b\in (0,1]$, let
\eqb \label{eqn-flow-path-scale}
E_b^\ep := \left\{ \sum_{e \in \mcl E\mcl G^\ep(\BB A_{(\ep/b)^\beta R_b , \rho R_b })} [\theta^\ep(e)]^2 \leq C_0 \log(b/\ep) \right\} .
\eqe 
By the discussion in the preceding paragraph and the definition of $\theta^\ep$, we can equivalently define $E_b^\ep$ to be the event of Lemma~\ref{lem-flow-path-macro} with $\ep/b$ in place of $\ep$ and the left field / curve pair in~\eqref{eqn-field-curve-scale} in place of $(h,\eta)$.  
By Lemma~\ref{lem-flow-path-macro} and~\eqref{eqn-field-curve-scale}, there exists $\alpha  = \alpha  (\gamma) \in (0,\alpha_0]$ such that
\eqb \label{eqn-annulus-path-scale-prob}
\BB P\left[ E_b^\ep \right] = 1 - O( (\ep/b)^\alpha) 
\eqe 
a rate depending only on $\rho$ and $\gamma$ as $\ep/b\to 0$.
\medskip

\noindent\textit{Step 2: A particular choice of scales.} 
We now apply the estimates~\eqref{eqn-annulus-path-scale-prob} and~\eqref{eqn-cone-hit-tail} at a carefully chosen set of scales. 
The key property which these scales will satisfy is~\eqref{eqn-cone-hit-cover} below. 

Recall that $Q =2/\gamma  + \gamma/2$ and $n\in\BB N_0$ let
\eqb \label{eqn-annulus-path-b-choice}
 b_n^\ep := \ep^{1 - (1-\xi)^n} \quad \op{for} \quad \xi := \frac{\beta\gamma (Q-\gamma)}{2}  \in (0,1) .
\eqe 
Then $b_0^\ep = 1$ and for $n\in\BB N$, $b_n^\ep := (\ep/b_{n-1}^\ep)^{\xi} b_{n-1}^\ep$. 
By Lemma~\ref{lem-cone-hit-tail} (applied with $C = \rho (\ep/b_{n-1}^\ep)^{-\beta/2}$), after possibly shrinking $\alpha$ (in a manner depending only on $\gamma$ and $\beta$) we can arrange that
\eqb \label{eqn-cone-hit-increment}
\BB P\left[ R_{b_n^\ep} \geq \frac{1}{\rho} (\ep / b_{n-1}^\ep)^{\beta} R_{b_{n-1}^\ep} \right] 
= 1 -  O( (\ep/b_{n-1})^{\alpha } ) 
\eqe 
at a rate depending only on $\gamma$.

For $T\in [1,1/\ep]$, let $n_T^\ep$ be the smallest $n\in\BB N$ with $b_n^\ep \leq T \ep$.
We note that $b_{n-1}^\ep \geq T\ep$, and so $b_n^\ep = \ep^\xi (b_{n-1}^\ep)^{1-\xi} \geq  T^{1-\xi}  \ep$.
By~\eqref{eqn-cone-hit-increment} and a union bound, after possibly further shrinking $\alpha$ we can arrange that with probability at least $1-O_T(T^{-\alpha})$, 
\eqb \label{eqn-cone-hit-cover}
 \left[ R_{T\ep}  , \rho \right]\subset \bigcup_{n=0}^{n_T^\ep} \left[  (\ep/b_n^\ep)^{\beta} R_{b_n^\ep}   ,    \rho R_{b_n^\ep}   \right]  . 
\eqe 
By~\eqref{eqn-annulus-path-scale-prob} (applied with $b = b_0^\ep ,\dots , b_{n_T^\ep}^\ep$) and a union bound, the event
\eqb \label{eqn-cone-hit-events}
  \bigcap_{n=0}^{n_T^\ep} E_{b_n^\ep}^\ep 
\eqe 
also occurs with probability at least $1-O_T(T^{-\alpha})$. 
\medskip

\noindent\textit{Step 3: Conclusion.} 
Suppose now that~\eqref{eqn-cone-hit-cover} and~\eqref{eqn-cone-hit-events} both hold, which happens with probability $1-O_T(T^{-\alpha})$. 
We will show that~\eqref{eqn-flow-path-macro} holds. To this end, we compute
\alb
\sum_{e \in \mcl E\mcl G^\ep(\BB A_{R_{T\ep} ,\rho})} [\theta^\ep(e)]^2
&\leq \sum_{n=0}^{n_T^\ep} \sum_{e \in \mcl E\mcl G^\ep\left(\BB A_{  (\ep/b_n^\ep)^{\beta} R_{b_n^\ep}   ,    \rho R_{b_n^\ep}     }\right)} [\theta^\ep(e)]^2 
\quad \text{(by~\eqref{eqn-cone-hit-cover})} \notag\\
&\leq C_0 \sum_{n=0}^{n_T^\ep} \log(b_n^\ep/\ep) 
\quad \text{(by~\eqref{eqn-cone-hit-events})} \notag\\
&\leq C_0 \log\ep^{-1} \sum_{n=0}^{n_T^\ep} (1-\xi)^n \preceq   \log\ep^{-1}
\quad \text{(by~\eqref{eqn-annulus-path-b-choice})} ,
\ale
with a deterministic implicit constant depending only on $\rho$ and $\gamma$. 
\end{proof}

Since we want a polylogarithmic bound for the probability of the event in Proposition~\ref{prop-eff-res-upper}, we need to apply Lemma~\ref{lem-flow-path-micro} with $T$ at least some positive power of $\log \ep^{-1}$. To deal with the energy of $\theta^\ep$ over $B_{R_{T\ep}}$ for such a value of $T$, we will use the following crude a priori bound.

\begin{lem} \label{lem-flow-path-tiny}
For each $\zeta>0$, there exists $\alpha=\alpha(\zeta,\gamma) > 0$ such that for $T\in [1,1/\ep)$,  
\eqb \label{eqn-flow-path-tiny}
\BB P\left[ \sum_{e \in \mcl E\mcl G^\ep(B_{R_{T\ep}})} [\theta^\ep(e)]^2 \leq   T^{2+\zeta} \right] =   1-O_T(T^{-\alpha}) ,
\eqe 
at a rate depending only on $\rho$ and $\gamma$. 
\end{lem}
\begin{proof}
We claim that there is an $\alpha=\alpha(\zeta,\gamma)>0$ such that with probability $1-O_T(T^{-\alpha})$, the graph $\mcl G^\ep(B_{R_{T\ep}}) $ has at most $T^{1+\zeta/2}$ vertices (and hence at most $T^{2+\zeta}$ edges). Since $|\theta^\ep(e)| \leq 1$ for each edge $e$ by definition, this claim immediately implies~\eqref{eqn-flow-path-tiny}.

We now prove the above claim. 
By the scaling property~\eqref{eqn-cone-scale}, the $\gamma$-LQG coordinate change formula, and standard estimates for the LQG measure (see, e.g.,~\cite[Lemma~A.3]{ghs-dist-exponent}), there exists $\alpha=\alpha(\zeta,\gamma)>0$ such that $\BB P[ \mu_h(B_{2 R_{T\ep}}) \leq T^{1+\zeta/4} \ep] = 1- O_T(T^{-\alpha})$. Since the cells $\eta([x-\ep,x])$ of $\mcl G^\ep$ have $\mu_h$-mass $\ep$, this implies that with probability $1-O_T(T^{-\alpha})$, there are at most $T^{1+\zeta/4}$ such cells which are completely contained in $B_{2R_{T\ep}}$.  By~\cite[Proposition~3.4]{ghm-kpz} and the scale invariance of the law of space-filling SLE, modulo time parameterization, it holds except on an event of probability decaying faster than any negative power of $T$ that each cell of $\mcl G^\ep$ which intersects both $ B_{R_{T\ep}}$ and  $\bdy B_{2R_{T\ep}}$ contains a Euclidean ball of radius at least $T^{-\zeta/4} R_{T\ep}$, so there can be at most $4\pi T^{-\zeta/2}$ such cells. 
Combining these estimates and slightly shrinking $\zeta > 0$ proves our claim.  
\end{proof}

\begin{proof}[Proof of Proposition~\ref{prop-eff-res-upper}]
By combining Lemmas~\ref{lem-flow-path-micro} and~\ref{lem-flow-path-tiny}, each applied with, e.g., $T = (\log \ep^{-1})^{1/4}$, we see that the energy of $\theta^\ep$ is at most $C\log\ep^{-1}$ with probability $1-O_\ep((\log\ep^{-1})^\alpha)$ for appropriate constants $C,\alpha$ as in the proposition statement. The desired effective resistance bound now follows from Thomson's principle~\eqref{eqn-thomson}. 
\end{proof}

\begin{remark} \label{remark-asaf}
In an earlier draft of this paper, we proved Proposition~\ref{prop-eff-res-upper} via a somewhat more complicated argument which proceeded by using~\cite[Theorem 3.2]{gms-harmonic} to prove a lower bound for the Dirichlet energy of the function on $\mcl V\mcl G^\ep(B_\rho)$ which is equal to 1 at 0, vanishes on $\mcl V\mcl G^\ep(\bdy B_\rho)$, and is discrete harmonic elsewhere; and then applying Dirichlet's principle~\eqref{eqn-eff-res-dirichlet}. We thank Asaf Nachmias for suggesting the idea for the alternative argument given here.
\end{remark}

\subsection{Proof of Proposition~\ref{prop-sg-eff-res}}
\label{sec-ball-eff-res}

Since $\mcl G^\ep \eqD \mcl G$ for $\ep > 0$, it suffices to show that for an appropriate choice of constants $C$ and $\alpha$ as in the statement of the lemma, there exists for each $r \in\BB N$ an $\ep = \ep_r > 0$ (depending on $r$) such that with probability at least $1-O_r((\log r)^{-\alpha})$, 
\eqb \label{eqn-sg-eff-res-upper}
   \mcl R^{\mcl G^\ep}\left(  0 \leftrightarrow \bdy \mcl B_r^{\mcl G^\ep}(0)    \right)   \leq   C  \log r  ; 
\eqe 
and a possibly different $r$-dependent choice of $\ep > 0$ such that with probability at least $1-O_r((\log r)^{-\alpha})$, 
\eqb \label{eqn-sg-eff-res-lower}
   \mcl R^{\mcl G^\ep}\left(  0 \leftrightarrow \bdy \mcl B_r^{\mcl G^\ep}(0 )      \right)   \geq   C^{-1}  \log r .
\eqe 

We know from Propositions~\ref{prop-eff-res-lower} and~\ref{prop-eff-res-upper} that there exists $C_0 = C_0(\gamma)  >0$ and $\alpha = \alpha(\gamma)  > 0$ such that for each $\ep \in (0,1)$, it holds with probability at least $1-O_\ep( (\log \ep^{-1})^{-\alpha})$ that 
\eqb \label{eqn-use-ball-eff-res}
C_0^{-1} \log \ep^{-1} \leq      \mcl R^{\mcl G^\ep}\left(  0 \leftrightarrow \mcl V\mcl G^\ep\left(\bdy B_{1/2}(0) \right)   \right)  \leq C_0 \log \ep^{-1}  .
\eqe 
So, we need to compare Euclidean balls and graph distance balls. 

We start by proving~\eqref{eqn-sg-eff-res-upper} for an appropriate choice of $\ep$. Fix $q\in \left(0 , \tfrac{2}{(2+\gamma)^2} \right)$, chosen in a manner depending only on $\gamma$. By Lemma~\ref{lem-max-cell-diam}, it holds except on an event of probability decaying faster than some positive power of $\ep$ that the Euclidean diameter of each cell of $\mcl G^\ep$ which intersects $B_{1/2}(0)$ is at most $\ep^q$. This implies any path in $\mcl G^\ep$ from $0$ to $\mcl V\mcl G^\ep(\bdy B_{1/2}(0))$ must contain at least $\frac12 \ep^{-q}$ cells of $\mcl G^\ep$, i.e., $\mcl V\mcl B_{\tfrac12 \ep^{-q}}(0 ; \mcl G^\ep) \subset \mcl V\mcl G^\ep( B_{1/2}(0) )$. Therefore,~\eqref{eqn-sg-eff-res-upper} with $\ep = (2 r)^{-1/q}$ and $C $ slightly larger than $q^{-1} C_0$ follows from the upper bound in~\eqref{eqn-use-ball-eff-res}. 

Now we establish the lower bound~\eqref{eqn-sg-eff-res-lower}. Since the cells of $\mcl G^\ep$ each have $\mu_h$-mass $\ep$, we infer from standard estimates for the $\gamma$-LQG measure (see, e.g.,~\cite[Lemma A.3]{ghs-dist-exponent}) that except on an event of probability decaying faster than some positive power of $\ep$, we have $\# \mcl V\mcl G^\ep(B_{1/2}(0)) \leq \ep^{-2}$, say, which means that $\mcl V\mcl G^\ep(B_{1/2}(0)) \subset \mcl V\mcl B_{\ep^{-2}}^{\mcl G^\ep}(0 )$. Combining this with the lower bound in~\eqref{eqn-use-ball-eff-res} yields~\eqref{eqn-sg-eff-res-lower} with $\ep = r^{-2}$ and $C = 2C_0$.  \qed

\section{Extension to other random planar maps}
\label{sec-discrete}

\subsection{Coupling with the mated-CRT map}
\label{sec-coupling}

In order to transfer from the estimates for mated-CRT maps in Section~\ref{sec-eff-res} to estimates for the other random planar map models we consider here, we will employ a coupling result for $M$ and $\mcl G$ which is proven in~\cite{ghs-map-dist} by means of the strong coupling of random walk with Brownian motion~\cite{zaitsev-kmt}.
Suppose $(M,\BB v)$ is one of the first five rooted random planar maps listed in Section~\ref{sec-main-results} and let $\BB e$ be an oriented root edge for $M$ with initial endpoint $\BB v$. If we equip $M$ with a certain statistical mechanics model $\mcl S$, then there is a bijective encoding of $(M,\BB e , \mcl S)$ (a so-called \emph{mating-of-trees} bijection) by means of a two-sided two-dimensional random walk $\mcl Z : \BB Z \to \BB Z^2$ with i.i.d.\ increments and a certain step distribution depending on the model.  
\begin{enumerate}
\item In the case of the UIPT of type II, $\mcl S$ is critical ($p=1/2$) site percolation on $M$, or equivalently a uniform depth-first-search tree on $M$.  This bijection is introduced in~\cite{bernardi-dfs-bijection} in the setting of a uniform depth-first-search tree on a finite triangulation. The paper~\cite{bhs-site-perc} explains the connection to site percolation and the (straightforward) extension to the UIPT.
\item In the case of the infinite spanning-tree decorated planar map, $\mcl S$ is a uniform spanning tree on $M$~\cite{mullin-maps,bernardi-maps,shef-burger}.
\item For infinite bipolar-oriented planar maps of various types, $\mcl S$ is a uniformly chosen orientation on the edges of $M$ with no source or sink (i.e., the source and sink are equal to~$\infty$)~\cite{kmsw-bipolar}.
\item For the uniform infinite Schnyder-wood decorated triangulation, $\mcl S$ is a uniformly chosen Schnyder wood on $M$~\cite{lsw-schnyder-wood}.
\end{enumerate}
These bijections are each reviewed in~\cite{ghs-map-dist}. We will not need the precise definitions of the bijections here. In fact, the particular statistical mechanics model on the map does not matter for our purposes --- we only need some model for which these exists a mating-of-trees bijection wherein the walk has i.i.d.\ increments.
 
In each of the above cases, we let $\mcl G$ be the $\gamma$-mated-CRT map where $\gamma$ is the LQG parameter corresponding to $M$ as listed in Section~\ref{sec-main-results}. To state our coupling result, we need to define for each $n\in\BB N$ a planar map which corresponds to the finite time interval $[-n,n]_{\BB Z}$ (the reason for this is that the coupling theorem of~\cite{zaitsev-kmt} only allows us to compare random walk and Brownian motion on a finite time interval). We start by considering the needed maps in the mated-CRT map case.

\begin{defn} \label{def-sg-restrict}
For $n\in\BB N$, we write $\mcl G_{n}$ for the sub-graph of $\mcl G$ whose vertex set is $[-n,n]_{\BB Z}$ and whose edge set consists of all of the edges of $\mcl G$ between two such vertices.  
\end{defn}

In each of the cases listed above, the corresponding bijection gives for each $n\in\BB N$ a planar map~$M_n$ with boundary\footnote{Recall that a planar map with boundary is a planar map $G$ with a distinguished face (the \emph{external face}), in which case the boundary $\bdy G$ is the set of vertices and edges on the boundary of this face.}~$\bdy M_{n}$ associated with the random walk increment $\mcl Z|_{[-n,n]_{\BB Z}}$. The map $M_{n}$ is the discrete analog of $\mcl G_n$ and is defined in a slightly different manner in each case (the particular definitions are given in~\cite{ghs-map-dist}).
 
The map $M_{n}$ is not necessarily a subgraph of $M$ since it is possible that some pairs of vertices or pairs of edges of $\bdy M_n$ can be identified together after time $n$ (see~\cite[Section 1.4]{ghs-map-dist} for further discussion). But, there is an (almost) inclusion map 
\eqb \label{eqn-inclusion-function}
\iota_n : M_{n} \to M \quad \text{which is injective on $M_{n} \setminus \bdy M_{n}$},
\eqe 
which means that we can canonically identify $M_{n} \setminus \bdy M_{n}$ with a subgraph of $M$. The map $M_n$ possesses a canonical root vertex which is mapped to $\BB v$ by $\iota_n$, and which (by a slight abuse of notation) we identify with $\BB v$.

One also obtains for each $n \in\BB N$ functions
\eqb \label{eqn-peano-functions}
\phi_n : \mcl V(M_{n}) \to [-n,n]_{\BB Z}  \quad \op{and} \quad \psi_n : [-n,n]_{\BB Z} \to \mcl V(M_{n}) 
\eqe    
which satisfy $ \phi_n(\BB v) =0$ and $ \psi_n(0) = \BB v$.
Roughly speaking, the vertex $\psi_n(i)$ corresponds to the $i$th step of the walk $\mcl Z $ in the bijective construction of $(M,e_0,T)$ from $\mcl Z$ and $\phi_n$ is ``close" to being the inverse of $\psi_n$. 
However, the construction of $M_{n}$ from $\mcl Z|_{[-n,n]_{\BB Z}}$ does not set up an exact bijection between $[-n, n]_{\BB Z}$ and the vertex set of $M_{n} $, so the functions $\phi_n$ and $\psi_n$ are neither injective nor surjective.  See Figure~\ref{fig-coupling-setup} for an illustration of the above definitions.

\begin{figure}[ht!]
 \begin{center}
\includegraphics[scale=.7]{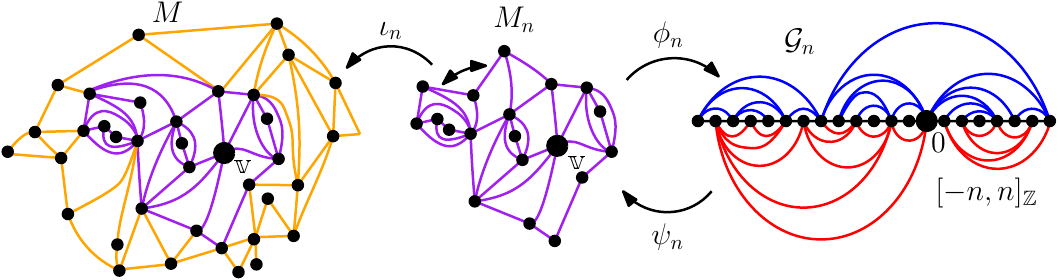} 
\caption{Illustration of the map $M_n$, the ``almost inclusion" map $\iota_n : M_n\rta M$, and the maps $\phi_n : \mcl V(M_n) \rta \mcl V(\mcl G_n)$ and $\psi_n : \mcl V(\mcl G_n) \rta \mcl V(M_n)$. Note that two edges of $\bdy M_n$ get identified when we apply $\iota_n$ in the figure. Theorem~\ref{thm-map-count} implies that the maps $\phi_n$ and $\psi_n$ are rough isometries up to a polylogarithmic factor, and that they distort discrete Dirichlet energies by at most a polylogarithmic factor (Lemma~\ref{lem-energy-compare}). A similar figure and caption appears in~\cite{ghs-map-dist}. 
}\label{fig-coupling-setup}
\end{center}
\end{figure}

Our main tool for comparing $M$ and $\mcl G$ is the following theorem, which is~\cite[Theorem~1.9]{ghs-map-dist}.
For the theorem statement, we define a \emph{path} in a graph $G$ be a function $P : [0,N]_{\BB Z} \to \mcl V(G)$ with the property that $P(i-1)$ and $P(i)$ are either identical or joined by an edge in $\mcl G$ for each $i\in [1,N]_{\BB Z}$. We write $|P| := N$ for the length of $P$.  
 
\begin{thm}
\label{thm-map-count}
Let $\phi_n$ and $\psi_n$ be as in~\eqref{eqn-peano-functions}. 
There are universal constants $p,q>0$ such that for each $A > 0$,  there is a $C = C(A) >0$ such that for each $n\in\BB N$, there is a coupling of $\mcl G$ and $(M,\BB e , \mcl S)$ such that with probability $1-O_n(n^{-A})$, the following is true. 
\begin{enumerate} 
\item For each $v_1,v_2 \in  \mcl V(M_{n} )$ with $v_1 \sim v_2$ in $M_n$, there is a path $P_{v_1,v_2}^{\mcl G}$ from $\phi_{n}(v_1)$ to $\phi_{n}(v_2)$ in $\mcl G_{n}$ with $|P_{v_1,v_2}^{\mcl G}| \leq C (\log n)^p$, and each $i\in [-n , n]_{\BB Z}$ is hit by a total of at most $ C (\log n)^q $ of the paths $P_{v_1,v_2}^{\mcl G}$ for $v_1,v_2 \in  \mcl V(M_{n} )$ with $v_1 \sim v_2$.  \label{item-map-count-G}
\item For each $x_1,x_2 \in    [-n , n]_{\BB Z}$ with $x_1 \sim x_2$ in $\mcl G_n$, there is a path $P_{x_1,x_2}^M$ from $\psi_n(x_1)$ to $\psi_n(x_2)$ in $M_{n}$ with $|P_{x_1,x_2}^{M}| \leq C (\log n)^p$, and each $v \in \mcl V(M_{n})$ is hit by a total of at most $ C (\log n)^q$ of the paths $P_{x_1,x_2}^M$ for $x_1,x_2 \in    [-n , n]_{\BB Z}$ with $x_1 \sim x_2$.  \label{item-map-count-M} 
\item We have $\op{dist}^{M_n}\left(\psi_{n}(\phi_{n}(v)) , v  \right) \leq C ( \log n)^p$ for each $v\in \mcl V(M_{n})$ and $ \op{dist}^{\mcl G_n}\left(\phi_{n}(\psi_{n}(x)) , x  \right) \leq C (\log n)^p $ for each $x \in [-n , n]_{\BB Z}$. \label{item-map-count-close}
\end{enumerate}
\end{thm}

As explained in~\cite[Lemma~1.10]{ghs-map-dist}, the conditions of Theorem~\ref{thm-map-count} imply that $\phi_n$ and $\psi_n$ are rough isometries up to a factor of $C(\log n)^p$, meaning that for each $v_1,v_2 \in \mcl V(M|_{n})$, 
\allb \label{eqn-map-coupling-G}
C^{-1} (\log n)^{-p} \op{dist}^{M_{n}}\left( v_1 , v_2    \right)  - 2 
 \leq  \op{dist}^{\mcl G_{n}}\left( \phi_n(v_1) , \phi_n(v_2)   \right) 
  \leq  C (\log n)^p \op{dist}^{M_{n}}\left( v_1 , v_2    \right)   ; 
\alle
and for each $x_1,x_2 \in [-n  , n ]_{\BB Z}$,
\allb \label{eqn-map-coupling-M}
 C^{-1} (\log n)^{-p} \op{dist}^{\mcl G_n} \left( x_1 , x_2     \right)  - 2
 \leq \op{dist}^{M_n}\left( \psi_n(x_1) , \psi_n(x_2)   \right)  
  \leq  C (\log n)^p \op{dist}^{\mcl G_n} \left( x_1 , x_2    \right)    . 
\alle 

Theorem~\ref{thm-map-count} also enables us to compare the Dirichlet energies of functions on $M_n$ and $\mcl G_n$ (recall Definition~\ref{def-discrete-dirichlet}). 

\begin{lem} \label{lem-energy-compare}
Fix $A > 0$ and suppose we have coupled $M$ with $\mcl G$ in such a way that the conclusion of Theorem~\ref{thm-map-count} is satisfied. 
If we let $p,q>0$ and $C = C(A)  > 0$ be the constants from that theorem, then for each $n\in\BB N$ it holds with probability $1-O_n(n^{-A})$ that the following is true (with $\phi$ and $\psi$ as in~\eqref{eqn-peano-functions}). 
Each function $f : [-n  , n ]_{\BB Z} \to \BB R$ satisfies
\eqb \label{eqn-energy-compare-G}
   \op{Energy}\left( f \circ \phi_n ; M_{n} \right)    \leq  C^2 (\log n)^{p+q}  \op{Energy}\left(f ; \mcl G_{n} \right)  
\eqe  
and each function $g : \mcl V( M|_{n}) \to \BB R$ satisfies
\eqb \label{eqn-energy-compare-M}
  \op{Energy}\left( g \circ \psi_n ; \mcl G_{n} \right)    \leq  C^2 (\log n)^{p+q}  \op{Energy}\left( g ; M_{n} \right) .
\eqe  
\end{lem}
\begin{proof}
Assume we are working on the event that the conclusion of Theorem~\ref{thm-map-count} is satisfied, which happens with probability $1-O_n(n^{-A})$. 
For an edge $\{v_1,v_2\} \in \mcl E(M_{n})$, let $P_{v_1,v_2}^{\mcl G}$ be the path from $\phi_n(v_1)$ to $\phi_n(v_2)$ in $\mcl G_{n}$ from Theorem~\ref{thm-map-count} which satisfies $|P_{v_1,v_2}^{\mcl G}| \leq C (\log n)^p$. 
By the Cauchy-Schwarz inequality, for any function $f : [-n  , n ]_{\BB Z} \to \BB R$, 
\alb
(f( \phi_n(v_1) ) - f( \phi_n(v_2) ) )^2 
&\leq \left( \sum_{i=1}^{|P_{v_1,v_2}^{\mcl G} |}  |f( P_{v_1,v_2}^{\mcl G}(i) )   - f( P_{v_1,v_2}^{\mcl G}(i-1)  ) | \right)^2 \notag\\
&\leq  C (\log n)^p \sum_{i=1}^{|P_{v_1,v_2}^{\mcl G}|}  \left( f( P_{v_1,v_2}^{\mcl G}(i) )   - f( P_{v_1,v_2}^{\mcl G}(i-1)  ) \right)^2  .
\ale
Summing over all such edges $\{v_1,v_2\}$ gives
\eqb \label{eqn-energy-compare-sum}
\op{Energy}\left( f \circ \phi_n  ; M_{n} \right) 
\leq  C (\log n)^p \sum_{\{v_1,v_2\} \in \mcl E (M) } \sum_{i=1}^{|P_{v_1,v_2}^{\mcl G}|} \left( f( P_{v_1,v_2}^{\mcl G}(i) )   - f( P_{v_1,v_2}^{\mcl G}(i-1)  ) \right)^2  .
\eqe 
By Theorem~\ref{thm-map-count}, each vertex $x\in [-n  , n ]_{\BB Z}$ is hit by at most $C (\log n)^q$ of the paths $P_{v_1,v_2}^{\mcl G}$. Consequently, each edge of $\mcl G_{n}$ is traversed by at most $C (\log n)^q$ of these paths, so the right side of~\eqref{eqn-energy-compare-sum} is bounded above by $C^2 (\log n)^{p+q}  \op{Energy}\left(f ; \mcl G_{n} \right) $. This gives~\eqref{eqn-energy-compare-G}. The bound~\eqref{eqn-energy-compare-M} is proven using exactly the same argument with the roles of $M$ and $\mcl G$ interchanged.
\end{proof}

\subsection{Proofs of main theorems}
\label{sec-proof}

We now conclude the proofs of our main theorems. 
We start by transferring the upper bound for effective resistance from Proposition~\ref{prop-sg-eff-res} to the other models considered in this paper.

\begin{prop} \label{prop-map-eff-res}
Suppose $(M,\BB v)$ is one of the first five random planar maps listed in Section~\ref{sec-main-results}. 
There exists $\alpha , C , p > 0$ (depending on the particular model) such that for each $r \in\BB N$, it holds with probability $1-O_r( (\log r)^{-\alpha})$ that
\eqb \label{eqn-map-eff-res}
   \mcl R^M\left(  \BB v \leftrightarrow \bdy \mcl B_r^M(\BB v  )  \right)   \leq   C (\log r)^p .
\eqe 
\end{prop}

The same argument also gives the lower bound $\mcl R^M\left(  \BB v \leftrightarrow \bdy \mcl B_r^M(\BB v  )  \right) \geq C^{-1} (\log r)^{-p}$, but this lower bound is trivial since by~\eqref{eqn-eff-res-def},
\eqbn
\mcl R^M\left(  \BB v \leftrightarrow \bdy \mcl B_r^M(\BB v  )  \right) \geq 1/\op{deg}^M(\BB v) .
\eqen 

\begin{proof}[Proof of Proposition~\ref{prop-map-eff-res}]
Let $\gamma \in (0,2)$ be the LQG parameter for $M$ and let $\mcl G$ be the $\gamma$-mated-CRT map. 
For $m\in\BB N$, let $\frk f_m : \BB Z \to [0,1]$ be the function such that $\frk f_{m}(0) = 1$, $\frk f_{m}$ vanishes outside of $  \mcl B_{m-1}^{\mcl G}(0 )$, and $\frk f_{m}$ is $\mcl G$-discrete harmonic on $\mcl B_{m-1}^{\mcl G}(0  ) \setminus \{0\}$. 
We similarly define $\frk g_m : \mcl V(M) \to [0,1]$ with $\BB v$ in place of $0$ and $M$ in place of $\mcl G$ throughout.  
By Dirichlet's principle~\eqref{eqn-eff-res-dirichlet},
\allb \label{eqn-map-eff-res-dirichlet}
&\op{Energy}\left( \frk f_{m} ;    \mcl G   \right)  = \mcl R^{\mcl G}\left(  0 \leftrightarrow \bdy \mcl B_{m}^{\mcl G}(0   )      \right)^{-1} 
\quad \op{and} \notag\\
&\qquad \qquad \op{Energy}\left( \frk g_{m} ;    M    \right)  = \mcl R^M\left(  \BB v \leftrightarrow \bdy \mcl B_{m}^M(\BB v  )      \right)^{-1}  .
\alle 
We will compare the Dirichlet energies of $\frk f_{r^2}$ and $\frk g_r$ using Lemma~\ref{lem-energy-compare}, then estimate the former using Proposition~\ref{prop-sg-eff-res}. 

By~\cite[Lemma~1.11]{ghs-map-dist}, there exists $K  = K (\gamma ) >1$ such that with probability $1-O_r(r^{-1})$, 
\eqb \label{eqn-use-ball-iso}
\mcl B_{r^2}^{\mcl G}(0 ) \subset \mcl G_{r^K} \quad \op{and} \quad \text{$\iota_{r^K}$ restricts to a graph isomorphism $\mcl B_r^{M_{r^K}}(\BB v  ) \to \mcl B_r^M(\BB v )$}
\eqe
where here we recall that $\iota_{r^K}$ is defined in~\eqref{eqn-inclusion-function}. 
Henceforth suppose we have coupled $M$ and $\mcl G$ together as in Theorem~\ref{thm-map-count} with $A = 1$ and $n = \lceil r^K \rceil$.
Let $C_1 = C (1) > 0$ be the constant from that theorem.
 
On the event~\eqref{eqn-use-ball-iso}, we can view $\frk g_r$ as a function on $\mcl V(M_{r^{K }})$ by identifying $\mcl B_r^{M_{r^K}}(\BB v  )$ and $\mcl B_r^M(\BB v )$ and recalling that   
$\frk g_r$ vanishes outside $\mcl B_r^M(\BB v) \subset M_{ r^{K } }$.
By Lemma~\ref{lem-energy-compare} and~\eqref{eqn-use-ball-iso}, there exists a universal constant $p_0 >0$ (equal to the constant $p+q$ from Lemma~\ref{lem-energy-compare}) such that with probability $1-O_r(r^{-1})$, 
\allb \label{eqn-map-eff-res-compare}
  \op{Energy}\left( \frk g_r \circ \psi_{r^{K } } ;  \mcl G \right)  \leq  C_1^2 (\log r)^{p_0} \op{Energy}\left( \frk g_r  ;  M \right)  .
\alle

By~\eqref{eqn-map-coupling-M} and~\eqref{eqn-use-ball-iso}, it holds with probability $1-O_r(r^{-1})$ that 
\eqb \label{eqn-map-eff-res-contain}
 \psi_{r^{K }}^{-1}\left( \mcl B_r^M (\BB v    ) \right) \subset \mcl B_{r^2}^{\mcl G}(0  )  .
\eqe 
Since $\psi_{r^{K } }(0) = \BB v$, if~\eqref{eqn-map-eff-res-contain} holds then the function $\frk g_r \circ \psi_{r^{K } }$ vanishes outside $\mcl B_{r^2}^{\mcl G}(0  )$ and equals $1$ at the origin. 
Since the discrete harmonic function $\frk f_{r^2}$ minimizes Dirichlet energy subject to specified boundary data, we infer from~\eqref{eqn-map-eff-res-contain} and~\eqref{eqn-map-eff-res-compare} that with probability $1-O_r(r^{-1})$, 
\allb \label{eqn-map-eff-res-compare'}
 \op{Energy}\left( \frk f_{r^2}  ;  \mcl G \right)  \leq  C_1^2 (\log r)^{p_0} \op{Energy}\left( \frk g_r  ;  M \right)  .
\alle 
We obtain~\eqref{eqn-map-eff-res} with $p=p_0+1$ by applying~\eqref{eqn-map-eff-res-dirichlet} and Proposition~\ref{prop-sg-eff-res} to bound $ \op{Energy}\left( \frk f_{r^2} ;  \mcl G \right)$ from below, then plugging the resulting estimate into~\eqref{eqn-map-eff-res-compare} and again using~\eqref{eqn-map-eff-res-dirichlet}.
\end{proof}

\begin{proof}[Proof of Theorem~\ref{thm-map-green}]
The bound~\eqref{eqn-map-eff-res0} for the Green's function at the exit time $\sigma_r$ from $\mcl B_r^M(\BB v)$ is immediate from Proposition~\ref{prop-map-eff-res}, the definition~\eqref{eqn-eff-res-def} of effective resistance, and the fact that $\op{deg}^M(\BB v)$ has an exponential tail. 
Since random walk on $M$ started from $\BB v$ cannot exit $\mcl B_n^M(\BB v  )$ before time $n$, 
\eqbn
\op{Gr}_n^M(\BB v , \BB v ) \leq \op{Gr}_{\sigma_n}^M(\BB v , \BB v)   .
\eqen 
Therefore, the desired bound~\eqref{eqn-map-green} for an appropriate choice of $\alpha, C$, and $p$ follows from~\eqref{eqn-map-eff-res0}.
\end{proof}

One similarly obtains from Proposition~\ref{prop-map-eff-res} the upper bound in~\eqref{eqn-green} of Theorem~\ref{thm-green}, which we re-state as the following lemma.

\begin{lem} \label{lem-green-upper}
For $\gamma \in (0,2)$, there exists $\alpha  = \alpha(\gamma) > 0$ and $C = C(\gamma) > 0$ such that for $n\in\BB N$, the $\gamma$-mated-CRT map satisfies
\eqbn
 \BB P\left[ \frac{ \op{Gr}_n^{\mcl G}(\BB v ,\BB v ) }{ \op{deg}^{\mcl G}\left( \BB v   \right)} \leq C \log n  \right] =   1- O_n\left( \frac{1}{(\log n)^\alpha} \right) .   
\eqen 
\end{lem}
\begin{proof}
This follows from the upper bound in Proposition~\ref{prop-sg-eff-res} and the fact that random walk on $\mcl G$ started from $0$ cannot exit $\mcl B_n^{\mcl G}(0)$ before time $n$.
\end{proof}

\begin{proof}[Proof of Theorems~\ref{thm-map-return} and~\ref{thm-map-spd}] 
To prove Theorem~\ref{thm-map-return}, we observe that~\cite[Proposition~10.18]{markov-mixing} shows that $j \mapsto \ol{\BB P}_{\BB v}^M \left[ X_{2j}^M = \BB v \right]$ is non-increasing. 
Hence,
\allb \label{eqn-spd-upper-mono}
\ol{\BB P}_{\BB v}^M \left[ X_{2n}^M = \BB v \right]
&\leq  \frac{1}{n} \sum_{j=1}^n \ol{\BB P}_{\BB v}^M \left[ X_{2j}^M = \BB v \right]
\leq \frac{1}{n} \ol{\BB E}_{\BB v}^M\left[ \# \left\{ j \in [1,2n]_{\BB Z} \,:\, X_j^M = \BB v \right\} \right] \notag\\
&= \frac{1}{n} \op{Gr}_{2n}^M (\BB v , \BB v ) .
\alle 
Combining this with Theorem~\ref{thm-map-green} yields~\eqref{eqn-map-return}. In the case when $M$ is a mated-CRT map, we see that we can take $p=1$ by using Lemma~\ref{lem-green-upper} in place of Theorem~\ref{thm-map-green}. 

As noted just before the statement of Theorem~\ref{thm-map-spd}, the return probability bound~\eqref{eqn-map-spd-bound} follows by combining Theorem~\ref{thm-map-return} and Lemma~\ref{lem-return-lower}. We now deduce from~\eqref{eqn-map-spd-bound} that the spectral dimension of $M$ is a.s.\ equal to 2. For $n\in\BB N$, let 
$P(n) :=   \ol{\BB P}_{\BB v}^M \left[ X_{2n}^M = \BB v \right]$. 
Also let $\alpha $ and $c$ be the constants from~\eqref{eqn-map-spd-bound}, fix $s > 1/\alpha$, and for $k\in\BB N$, let $n_k := \exp(k^s)$.  
By~\eqref{eqn-map-spd-bound} and the Borel-Cantelli lemma, it is a.s.\ the case that 
\eqb \label{eqn-spd-subsequence}
-2 \lim_{k\to\infty} \frac{ \log P(n_k)}{\log n_k} = 2 . 
\eqe  
By~\cite[Proposition~10.18]{markov-mixing}, if $n\in [n_k , n_{k+1}]$ then $P(n_k) \leq P(n) \leq P(n_{k+1})$. 
Since also $\lim_{k\to\infty} ( \log n_k ) / (\log n_{k+1})  =1 $, the convergence~\eqref{eqn-spd-subsequence} implies that
$-2 \lim_{n\to\infty}  \log P(n)/\log n = 2 $.
\end{proof}

It remains to prove our displacement lower bound (Theorem~\ref{thm-map-displacement}) and the lower bound for the Green's function on $\mcl G$ from Theorem~\ref{thm-green}.  For each of these two proofs, we will need the following degree bound.

\begin{lem} \label{lem-map-max-deg}
If $(M,\BB v)$ is any one of the five random planar maps considered in Section~\ref{sec-main-results}, 
then for each $A>0$, there exists $C   > 0$ (depending on the particular model) such that 
\eqb \label{eqn-map-max-deg}
\BB P\left[ \max_{v\in \mcl V \mcl B_r^M(\BB v ) } \op{deg}^M\left( v   \right) \leq C \log r \right] = 1 - O_r(r^{-A}) .
\eqe 
\end{lem}
\begin{proof} 
Since $\op{deg}^M(\BB v   )$ has an exponential tail and the walk (or the Brownian motion in the case of the mated-CRT map) which encodes $M$ has stationary increments, a union bound over all vertices of $M_{r^K}$ shows that for each $A  , K > 0$ that there exists $C = C(A,K  ) > 0$ such that 
\eqb \label{eqn-map-max-deg0}
\BB P\left[ \max_{v\in \mcl V( M_{r^K} )} \op{deg}^M\left( v    \right) \leq C \log r \right] = 1 - O_r(r^{-A}) .
\eqe 
The statement of the lemma follows by combining this with~\cite[Lemma~1.11]{ghs-map-dist}. 
\end{proof}

The following lemma will allow us to deduce the lower bound for displacement~\eqref{eqn-map-exit-prob} from our upper bound for return probability. 

\begin{lem} \label{lem-map-spd-compare}
If $(M,\BB v)$ is any one of the five random planar maps considered in Section~\ref{sec-main-results}, then there exists $\alpha = \alpha(\gamma) > 0$ and $C = C(\gamma)  > 0$ such that for each $r \in\BB N$, it holds with probability $1-O_r( r^{-\alpha})$ that
\eqb
\label{eqn-map-spd-compare}
 \ol{\BB P}_{\BB v}^M\left[ X^M_n  \in \mcl B_r^M(\BB v) \right]^2 \leq  C ( \log r) (\# \mcl V \mcl B_r^M(\BB v) )  \ol{\BB P}_{\BB v}^M\left[ X_{2n}^M = \BB v \right] ,\quad \forall n \in \BB N .
\eqe
\end{lem}
\begin{proof} 
By Lemma~\ref{lem-map-max-deg}, there exists $C_0 =C_0(\gamma) > 0$ such that with probability $1-O_r(r^{-1})$, 
\eqb \label{eqn-map-spd-lower-deg}
\max_{v\in \mcl V\mcl B_r^M(\BB v )   } \op{deg}^M\left( v  \right)  \leq C  \log r .
\eqe 
We now employ a classical calculation for random walk based on reversibility and the Cauchy-Schwarz inequality to find that if~\eqref{eqn-map-spd-lower-deg} holds then for $r\in\BB N$,  
\allb \label{eqn-map-cauchy-schwarz-reverse} 
\ol{\BB P}_{\BB v}^M\left[ X_{2n}^M  = \BB v \right] 
&\geq \sum_{v \in \mcl V \mcl B_r^M(\BB v ) } \ol{\BB P}_{\BB v}^M\left[ X_n^M  = v \right] \ol{\BB P}_{ v}^M\left[ X_n^M  = \BB v  \right] \notag \\
&=  \op{deg}^M\left( \BB v  \right)  \sum_{v \in \mcl V\mcl B_r^M(\BB v  )  } \frac{ \ol{\BB P}_{\BB v}^M\left[ X_n^M  = v  \right]^2  }{  \op{deg}^M\left(v    \right)   }  \quad\text{(by reversibility)}\notag\\
&\geq \left(  \sum_{v\in \mcl V\mcl B_r^M(\BB v )  } \op{deg}^M(v   ) \right)^{-1}  \left(   \sum_{v\in \mcl V\mcl B_r^M(\BB v  )   }  \ol{\BB P}_{\BB v}^M\left[ X_n^M  = v  \right] \right)^2  \notag\\
&\qquad \qquad \qquad \qquad \text{(by Cauchy-Schwarz)}\notag\\ 
&\geq  \frac{ \ol{\BB P}_{\BB v}^M\left[ X_n^M  \in   \mcl B_r^M(\BB v  )   \right]^2  }{ C  (\log r) (\# \mcl V  \mcl B_r^M(\BB v ) )            }   ,
\alle 
which gives~\eqref{eqn-map-spd-compare}. 
\end{proof}

\begin{proof}[Proof of Theorems~\ref{thm-map-displacement} and Theorem~\ref{thm-uipt-displacement}]
We first prove the upper bound~\eqref{eqn-map-exit-mean} for the expected exit time $\sigma_r$ from $\mcl B_r^M(\BB v)$. 
Using the reversibility of the Green's function $\op{Gr}_{\sigma_r}^M(\cdot, \cdot)$~\cite[Exercise~2.1]{lyons-peres}, we find that for each vertex $v$ of $\mcl B_r^M(\BB v  )$, 
\allb
\op{Gr}_{\sigma_r}^M(\BB v , v) 
&= \frac{\op{deg}^M(v  )}{\op{deg}^M(\BB v )} \op{Gr}_{\sigma_r}^M(v , \BB v  ) \notag\\
&= \frac{\op{deg}^M(v  )}{\op{deg}^M(\BB v )} \ol{\BB P}_v^M \left[ \text{$X^M$ hits $\BB v$ before exiting $\mcl B_r^M(\BB v   )$} \right] \op{Gr}_{\sigma_r}^M(\BB v , \BB v) \notag\\
&\leq \op{deg}^M(v ) \frac{\op{Gr}_{\sigma_r}^M(\BB v , \BB v) }{\op{deg}^M(\BB v  )} = \op{deg}^M(v )\mcl R^M\left(  \BB v \leftrightarrow \bdy \mcl B_r^M(\BB v )    \right) . \label{eqn-g-off-diagonal}
\alle
By summing the inequality~\eqref{eqn-g-off-diagonal} over each vertex $v$ of $\mcl B_r^M(\BB v  )$, we consequently see that
\eqb \label{eqn-map-ball-sum}
\ol{\BB E}_{\BB v}^M \left[ \sigma_r \right] 
\leq \mcl R^M\left(  \BB v \leftrightarrow \bdy \mcl B_r^M(\BB v )   \right) \sum_{v\in \mcl V\mcl B_r^M(\BB v ) } \op{deg}^M(v  ) .
\eqe
By applying Proposition~\ref{prop-map-eff-res} and Lemma~\ref{lem-map-max-deg} to bound the right side of~\eqref{eqn-map-ball-sum}, we obtain that for an appropriate choice of $C$ and $p$,~\eqref{eqn-map-exit-mean} holds with probability $1-O_r((\log r)^{-\alpha})$. 

To prove~\eqref{eqn-map-exit-prob}, we combine Theorem~\ref{thm-map-return} and Lemma~\ref{lem-map-spd-compare} to get that for appropriate constants $\alpha > 0$, $C >0$, and $p>1$, it holds for each $n , r\in\BB N$ that with probability $1-O_r((\log r)^{-\alpha}) - O_n( (\log n)^{-\alpha})$,   
\eqb \label{eqn-use-map-spd-compare}
 \ol{\BB P}_{\BB v}^M\left[ X^M_n  \in \mcl B_r^M(\BB v) \right]^2  \leq  \frac{C}{n} (\log n)^{p } ( \log r)  (\# \mcl V  \mcl B_r^M(\BB v) )     .
\eqe 
We now choose $r_n \to\infty$ such that $\BB P[ \# \mcl V  \mcl B_{r_n}^M(\BB v) \leq n (\log n)^{-p-1} ] \to 1$ as $n\to\infty$, as in the theorem statement, and note that we can take $r_n$ to grow faster than some positive power of $n$ due to~\cite[Theorem 1.10]{ghs-dist-exponent} (in the case of the mated-CRT map) or~\cite[Theorem 1.6]{ghs-map-dist} (in the case of other maps). Plugging this choice of $r_n$ into~\eqref{eqn-use-map-spd-compare} gives~\eqref{eqn-map-exit-prob} (with a possibly larger choice of $C$ and with $p+1$ in place of $p$). 

Theorem~\ref{thm-uipt-displacement} follows from Theorem~\ref{thm-map-displacement}, the fact that $\BB P[ \# \mcl V \mcl B_r^M(\BB v) \leq (\log r)^7 r^4  ] \to 1$ as $r\to\infty$ in the UIPT case~\cite[Theorem 1.2]{angel-peeling}, and the fact that $\BB P[ \# \mcl V \mcl B_r^M(\BB v) \leq (\log r)^q r^4 ] \to 1$ as $r\to\infty$ for some $q>0$ in the mated-CRT map case with $\gamma=\sqrt{8/3}$ (this follows from~\cite[Theorem 1.9 and Lemma 1.11]{ghs-map-dist} together with the analogous statement in the UIPT case). 
\end{proof}

\begin{proof}[Proof of Theorem~\ref{thm-green}]
The estimate~\eqref{eqn-eff-res0} is immediate from Proposition~\ref{prop-sg-eff-res}. 
The upper bound for the Green's function in~\eqref{eqn-green} is proven in Lemma~\ref{lem-green-upper}, so we just need to show that for an appropriate $\alpha > 0$ and $C>1$ depending only on $\gamma$, one has
\eqb \label{eqn-green-lower}
\BB P\left[  \frac{ \op{Gr}_n^{\mcl G}(0,0 ) }{ \op{deg}^{\mcl G}\left(0  \right)} \geq \frac{1}{C} \log n   \right] =   1- O_n\left( \frac{1}{(\log n)^\alpha} \right) .   
\eqe

To this end, let $\sigma_r$ for $r\in\BB N$ be the exit time of the simple random walk $X^{\mcl G}$ from $\mcl B_r^{\mcl G}(0)$, as in~\eqref{eqn-exit-time-def}.
By Proposition~\ref{prop-sg-eff-res}, there exists $C_0=C_0(\gamma) > 0$ such that for $r \in \BB N$ it holds except on an event of probability decaying faster than some positive ($\gamma$-dependent) power of $(\log r)^{-1}$ that
\eqb \label{eqn-exit-green-lower}
\op{Gr}_{\sigma_r}^{\mcl G}(0,0)   
\geq \frac{1}{C_0} \op{deg}^{\mcl G}(0) \log r .
\eqe
By the strong Markov property, under $\ol{\BB P}_0^{\mcl G}$ the number of times that $X^{\mcl G}$ returns to 0 before time $\sigma_r$ is a geometric random variable with mean $\op{Gr}_{\sigma_r}^{\mcl G}(0,0) $, so there is a constant $p=p(\gamma) > 0$ such that whenever~\eqref{eqn-exit-green-lower} holds,
\eqb \label{eqn-hit-count-lower}
\ol{\BB P}_0^{\mcl G} \left[ \text{$X^{\mcl G}$ returns to $0$ at least $ \op{deg}^{\mcl G}(0) \log r$ times before time $\sigma_r$} \right] 
\geq p.
\eqe 

By~\cite[Corollary 3.2]{ghs-dist-exponent}, there exits $\beta =\beta(\gamma) > 0$ small enough so that with probability $1-O_n(n^{-1})$, we have $\# \mcl V  \mcl B_{n^\beta}^{\mcl G}(0) \leq n^{1/2}$.
This estimate together with Theorem~\ref{thm-map-return} shows that the right side of~\eqref{eqn-map-spd-compare} with $r = n^\beta$ is at most $n^{-1/2+o_n(1)}$ with probability at least $1-O_n(n^{-1})$. By this and Lemma~\ref{lem-map-spd-compare}, we get that with probability $1-O_n(n^{-1})$,
\eqb \label{eqn-use-walk-lower}
\ol{\BB P}_0^{\mcl G}\left[  \sigma_{n^\beta} \leq n  \right] \geq \ol{\BB P}_0^{\mcl G}\left[  X^{\mcl G}_n \notin \mcl B_{n^\beta}^{\mcl G}(0)  \right]  \geq 1 -  \frac{p}{2}.
\eqe 
Combining~\eqref{eqn-hit-count-lower} (applied with $r = n^\beta$) and~\eqref{eqn-use-walk-lower} shows that with probability $1-O_n( (\log n)^{-\alpha})$, the $\ol{\BB E}_0^{\mcl G}$-expected number of times that $X^{\mcl G}$ hits $0$ before time $n$ is at least $(p/2) \beta \op{deg}^{\mcl G}(0) \log n  $, which gives~\eqref{eqn-green-lower} with $C = ((p/2)\beta)^{-1}$. 
\end{proof}

\appendix

\section{Lower bound for return probability on non-simple maps}
\label{sec-return-lower}

The works \cite{lee-conformal-growth,lee-uniformizing} by Lee prove a lower bound for the return probability to the root vertex for random walk on a local limit of finite random planar maps \emph{without multiple edges or self-loops} whose root vertex degree has an exponential tail. Some of the maps we consider in this paper are allowed to have multiple edges and/or self-loops, so here we explain why the results of~\cite{lee-conformal-growth} extend to this case.

\begin{lem} \label{lem-return-lower}
Suppose $(M,\BB v)$ is the Benjamini-Schramm limit of finite random planar maps with a uniformly random root vertex and that $\op{deg}^M(\BB v)$ has an exponential tail (e.g., $(M,\BB v)$ is one of the rooted planar maps considered in Section~\ref{sec-main-results}). There are constants $\alpha, p > 0$, depending on the particular law of $(M,\BB v)$ such that 
\eqb \label{eqn-return-lower}
\BB P\left[ \ol{\BB P}_{\BB v}^M\left[ X_{2n}^M = \BB v \right] \geq \frac{1}{n (\log n)^p} \right] = 1-O_n((\log n)^{-\alpha}) .
\eqe  
\end{lem}
\begin{proof}
If $M$ has no self-loops or multiple edges, then the lemma is a special case of~\cite[Theorem~1.7]{lee-conformal-growth} or~\cite[Theorem~1.6]{lee-uniformizing}. 
We will now treat the case when $M$ has multiple edges, but no self-loops, by considering the following two perturbations of $(M,\BB v)$:
\begin{enumerate}
\item Let $(M' , \BB v')$ be sampled from the law of $(M,\BB v)$ weighted by $1 + \frac{1}{2} \op{deg}^M(\BB v)$. 
\item Let $\wh M$ be the random planar map obtained from $M'$ by adding a vertex to the middle of each edge, i.e., we replace each edge by two edges which share a vertex. We identify $\mcl V(M')$ with the corresponding subset of $\mcl V(\wh M)$.  With probability $2/(2+\op{deg}^{M'}(\BB v'))$, let $\wh{\BB v} = \BB v'$ and with probability $\op{deg}^{M'}(\BB v')/(2+\op{deg}^{M'}(\BB v'))$, let $\wh{\BB v}$ be sampled uniformly from the set of $\op{deg}^{M'}(\BB v')$ neighbors of $\BB v'$ in $\wh M$.
\end{enumerate}
Since $\op{deg}^M(\BB v)$ has an exponential tail and by H\"older's inequality, it suffices to prove~\eqref{eqn-return-lower} with $(M',\BB v')$ in place of $(M ,\BB v)$. 
Furthermore, since we are assuming that $M$ has no self-loops, the map $\wh M$ has no self-loops or multiple edges. 

We first claim that~\eqref{eqn-return-lower} holds with $(\wh M , \wh{\BB v})$ in place of $(M,\BB v)$. Since $\op{deg}^M(\BB v)$ has an exponential tail, so does $\op{deg}^{\wh M}(\wh{\BB v})$.  By the aforementioned results of~\cite{lee-uniformizing,lee-conformal-growth}, to prove our claim we only need to show that $(\wh M , \wh{\BB v})$ is the Benjamini-Schramm limit of finite random planar maps with a uniform random root vertex.

To this end, let $\{M_n\}_{n\in\BB N}$ be a sequence of finite random planar maps which converge in law to $(M,\BB v)$ in the Benjamini-Schramm topology when equipped with a uniform random root vertex. Also let $\wh M_n$ for $n\in\BB N$ be obtained from $M_n$ by adding a vertex to the middle of each edge as in the definition of $\wh M$.
Let $\wh{\BB v}_n $ be sampled uniformly from $\mcl V(\wh M_n)$. If $\wh{\BB v}_n \in \mcl V(M_n)$, let $\BB v_n'=\wh{\BB v}_n$ and otherwise sample $\BB v_n'$ uniformly from the two neighbors of $\wh{\BB v}_n$. Then for $v\in\mcl V(M_n)$,  
\alb
\BB P\left[\BB v_n' =v \,|\, M_n \right] 
&= \BB P\left[  \wh{\BB v}_n =v \,|\, M_n \right] + \frac12 \BB P\left[\text{$\wh{\BB v}_n \sim v$ in $\wh M_n$} \,|\, M_n \right] \notag\\
&= \frac{1}{\#\mcl V(\wh M_n)  } \left( 1 + \frac12 \op{deg}^{M_n}(v) \right) .
\ale
Therefore, $(M_n ,\BB v_n') \rta (M',\BB v')$ in law with respect to the Benjamini-Schramm topology.
On the other hand, by Bayes' rule the conditional law of $\wh{\BB v}_n$ given $(M_n,\BB v_n')$ can be recovered by setting $\wh{\BB v}_n = \BB v_n'$ with probability $2/(2  + \op{deg}^{M_n}(\BB v_n') )$ and sampling $\wh{\BB v}_n'$ uniformly from the $\op{deg}^{M_n}(\BB v_n')$ neighbors of $\BB v_n'$ in $\wh M_n$ with probability $\op{deg}^{M_n}(\BB v_n')/(2 + \op{deg}^{M_n}(\BB v_n') )$. Therefore $(\wh M_n , \wh{\BB v}_n) \rta (\wh M , \wh{\BB v})$ in law with respect to the Benjamini-Schramm topology, as required.

It is easy to see that~\eqref{eqn-return-lower} for $(\wh M ,\wh{\BB v})$ implies the analogous estimate with $(\wh M , \BB v')$ in place of $(\wh M ,\wh{\BB v})$. 
If $X^{\wh M}$ is a simple random walk on $\wh M$ started from $\BB v'$ and we set 
\eqbn
\wh X^{M'}_m := X^{\wh M}_{2m} , \quad \forall m\in\BB N
\eqen
 then $\wh X^{M'}$ is a lazy random walk on $M'$. That is, $\wh X^{M'}$ takes a uniform nearest-neighbor step with probability $1/2$ or stays put with probability $1/2$. 
 
For $m\in\BB N$, let $\frk t(m)$ be the number of non-stationary steps for $\wh X^{M'}$ before time $m$. 
Also let $\frk j(n)$ for $n\geq 0$ be the time of the $n$th non-stationary step for $\wh X^{M'}$, so that $X^{M'}_n := \wh X^{M'}_{\frk j(n)}$ is a simple random walk on $M'$ and $X^{M'}$ is independent from $m\mapsto \frk t(m)$.

By Hoeffding's inequality for Bernoulli sums,  
\eqb \label{eqn-use-hoeff}
\ol{\BB P}_{\BB v'}^{M'}\left[ (1/2-\ep)m  \leq  \frk t(m) \leq  (1/2+\ep)m  \right] \geq 1 - 2 e^{-2\ep^2 m} ,\quad \forall m\in\BB N ,\quad \forall \ep  > 0. 
\eqe 
By~\eqref{eqn-return-lower} for $(\wh M, \BB v')$, for an appropriate constant $c>0$ (depending only on the law of $(M,\BB v)$) it holds with probability $1-O_m((\log m)^{-\alpha})$ that 
\allb \label{eqn-return-calc}
\frac{c}{m (\log m)^p} 
&\leq \ol{\BB P}_{\BB v'}^{M'}\left[ \wh X_{5m}^{M'} = \BB v' \right] \notag \\
&= \sum_{n = 2m}^{ 3m} \ol{\BB P}_{\BB v'}^{M'}\left[ \wh X_{5m}^{M'} = \BB v' ,\, \frk t(5m) = n \right]  + 2 e^{-c m} \quad\text{(by~\eqref{eqn-use-hoeff})}  \notag\\
&=  \sum_{n = 2m}^{3 m} \ol{\BB P}_{\BB v'}^{M'}\left[ \wh X_{\frk j(n)}^{M'} = \BB v' \,|\, \frk t(5m) = n \right] \BB P\left[ \frk t(5m) = n \right]  + 2 e^{-c m}  \notag \\
&= \sum_{n = 2m}^{3 m}   \ol{\BB P}_{\BB v'}^{M'}\left[   X_{n}^{M'} = \BB v'  \right] \BB P\left[ \frk t(5m) = n \right]   + 2 e^{-c m} \notag\\
&\qquad\qquad \qquad\qquad \text{(by independence of $X^{M'}$ and $\frk t$)} \notag \\
&\leq \max_{n \in [2m,3m]_{\BB Z}} \ol{\BB P}_{\BB v'}^{M'}\left[   X_{n}^{M'} = \BB v'    \right]   + 2 e^{-c m} .
\alle
By~\cite[Proposition 10.18]{markov-mixing}, one has
\[
\ol{\BB P}_{\BB v'}^{M'}\left[   X_{2n+2}^{M'} = \BB v'    \right]  \leq  \ol{\BB P}_{\BB v'}^{M'}\left[   X_{2n }^{M'} = \BB v'    \right], \quad \forall n\in\BB N.
\]
In fact, exactly the same proof shows that 
\[ 
\ol{\BB P}_{\BB v'}^{M'}\left[   X_{2n+1}^{M'} = \BB v'    \right]  \leq  \ol{\BB P}_{\BB v'}^{M'}\left[   X_{2n }^{M'} = \BB v'    \right], \quad\forall n \in \BB N.
\]
Consequently, the right side of~\eqref{eqn-return-calc} is at most $\ol{\BB P}_{\BB v'}^{M'}\left[   X_{2m}^{M'} = \BB v'    \right]+ 2 e^{-c m}  $, so~\eqref{eqn-return-lower} holds with $(M',\BB v')$ in place of $(M,\BB v)$ (with a slightly smaller choice of $p$). This completes the proof in the case when $M$ has multiple edges, but no self-loops.

If $M$ is allowed to have self-loops and multiple edges, then the map $\wh M$ defined above has multiple edges but not self-loops. Hence we can repeat the above argument verbatim to deduce the general case from the case of maps with no self-loops.
\end{proof}

\section{Index of notation}
\label{sec-index}

Here we record some commonly used symbols in the paper, along with their meaning and the location where they are first defined. Other symbols not listed here are only used locally. 

\begin{multicols}{2}
\begin{itemize}
\item $\gamma$: LQG parameter; Section~\ref{sec-overview}.
\item $Z = (L,R)$: correlated Brownian motion used to construct $\mcl G $; \eqref{eqn-bm-cov}.
\item $\mcl G$: mated-CRT map; \eqref{eqn-inf-adjacency}. 
\item $X^M$: random walk on $M$; Definition~\ref{def-rw-law}.
\item $\ol{\BB P}_v^M$: conditional law given $\mcl G $ of $X^M$ started from $v \in \mcl V(M)$; Definition~\ref{def-rw-law}.
\item $\op{dist}^M(\cdot,\cdot)$; graph distance on $M$; Definition~\ref{def-metric-ball}.
\item $\mcl B_r^M(\cdot)$; graph metric ball of radius $r$; Definition~\ref{def-metric-ball}.
\item $\op{Gr}_n^M(\cdot,\cdot)$; Green's function of $X^M$ stopped at time $n$; Section~\ref{sec-main-results}.
\item $\sigma_r$: hitting time of $\bdy \mcl B_r^M(\BB v)$; \eqref{eqn-exit-time-def}.
\item $\mcl V(G)$ and $\mcl E(G)$: vertex and edge sets of graph $G$; Section \ref{sec-basic-notation}.
\item $B_r(z)$: Euclidean ball of radius $r$ and center $z$ ($B_r = B_r(0)$); Section~\ref{sec-basic-notation}.
\item $\op{deg}^G(\cdot)$ degree of a vertex; Section \ref{sec-basic-notation}.
\item $\op{Energy}(\cdot ; \cdot)$: discrete Dirichlet energy; Definition~\ref{def-discrete-dirichlet}.
\item $\mcl R^M(\cdot \leftrightarrow \cdot)$: effective resistance; \eqref{eqn-eff-res-def}.
\item $Q= 2/\gamma + \gamma/2$: LQG coordinate change constant; \eqref{eqn-lqg-coord}. 
\item $h$: $\gamma$-quantum cone; Section \ref{sec-lqg-prelim}.
\item $\mu_h$: $\gamma$-LQG area measure; Section \ref{sec-lqg-prelim}. 
\item $\nu_h$: $\gamma$-LQG boundary measure; Section \ref{sec-lqg-prelim}.
\item $h_r(z)$: circle average of $h$ over $\bdy B_r(z)$;~\cite[Section 3.1]{shef-kpz}. 
\item $R_b$: largest radius with $h_r(\BB z) + Q\log r = \frac{1}{\gamma}\log b$; \eqref{eqn-mass-hit-time}.
\item $\kappa$: $16/\gamma^2$, SLE parameter; Section~\ref{sec-peanosphere}.
\item $\eta$: space-filling SLE$_\kappa$; Section \ref{sec-peanosphere}.   
\item $\mcl G^\ep$: mated-CRT map with cell size $\ep$; \eqref{eqn-inf-adjacency-ep}.
\item $\mcl G^\ep(D)$: subgraph of $\mcl G^\ep$ corresponding to domain $D\subset \BB C$; \eqref{eqn-sg-domain}.
\item $\BB A_{a,b}(z)$: the open annulus $B_b(z) \setminus \ol{B_a(z)}$ ($\BB A_{a,b} = \BB A_{a,b}(0)$); \eqref{eqn-annulus-def}.
\item $M_n$ and $\mcl G_n$: graphs corresponding to the time interval $[-n,n]_{\BB Z}$; Section~\ref{sec-coupling}. 
\item $\phi_n$ and $\psi_n$: rough isometry functions $M_n\rta\mcl G_n$ and $\mcl G_n\rta M_n$; \eqref{eqn-peano-functions}.
\end{itemize}
\end{multicols}

\bibliography{cibiblong,cibib}
\bibliographystyle{hmralphaabbrv}

\end{document}